\documentclass[11pt]{article}
\usepackage{amsmath, latexsym, amsfonts, amssymb, amsthm, amscd}
\usepackage{mathtools} 
\usepackage[left=2cm, right=2cm, top=2.5cm, bottom=2.5cm]{geometry}
\usepackage{upquote} 
\usepackage{color}
\usepackage{setspace} 
\usepackage[labelfont=bf]{caption} 
\usepackage{stmaryrd} 
\usepackage{multicol} 
\usepackage[dvipsnames]{xcolor}
\usepackage{comment}
\usepackage[title]{appendix}
 
\usepackage[utf8]{inputenc}
\usepackage[T1]{fontenc}
\usepackage{kpfonts} 

\usepackage{dsfont} 

\usepackage[english]{babel} 

\usepackage{bm} 

\usepackage{enumitem} 

\usepackage{xcolor}
\definecolor{Maroon}{HTML}{ad2231}
\definecolor{webgreen}{HTML}{008000}
\usepackage{dsfont}
\usepackage{hyperref}
\hypersetup{colorlinks, breaklinks, urlcolor=Maroon, linkcolor=Maroon, citecolor=webgreen} 

\allowdisplaybreaks



\newtheorem{theorem}{Theorem}[section]
\newtheorem{corollary}[theorem]{Corollary}
\newtheorem{proposition}[theorem]{Proposition}
\newtheorem{lemma}[theorem]{Lemma}
\newtheorem{remark}[theorem]{Remark}
\newtheorem{definition}[theorem]{Definition}

\theoremstyle{definition}

\usepackage{todonotes}

\usepackage{graphicx}
\newcommand{\cev}[1]{\reflectbox{\ensuremath{\vec{\reflectbox{\ensuremath{#1}}}}}}

\usepackage{chngcntr}
\usepackage{apptools}
\AtAppendix{\counterwithin{lemma}{section}}

\numberwithin{equation}{section} 

\begin{document}
\title{On the excision of Brownian bridge paths}
\author{Gabriel Berzunza Ojeda\footnote{ {\sc Department of Mathematical Sciences, University of Liverpool. Liverpool, United Kingdom.} E-mail: gabriel.berzunza-ojeda@liverpool.ac.uk}\,\, 
and
Ju-Yi Yen\footnote{{\sc Department of Mathematical Sciences, University of Cincinnati, Cincinnati, Ohio, United States.} E-mail: ju-yi.yen@uc.edu} \\ \vspace*{10mm}
}
\date{ }
\maketitle



\begin{abstract} 
Path transformations are fundamental to the study of Brownian motion and related stochastic processes, offering elegant constructions of the Brownian bridge, meander, and excursion. Central to this theory is the well-established link between Brownian motion and the $3$-dimensional Bessel process ${\rm BES}(3)$. This paper is specifically motivated by Pitman and Yor (2003), who showed that a ${\rm BES}(3)$ process can be constructed by excising the excursions of a Brownian path below its past maximum that reach zero and concatenating the remaining excursions. Our main result shows that a similar excision procedure, when applied to a Brownian bridge, can be related to a $3$-dimensional Bessel bridge. 
\end{abstract}




\section{Introduction}

Path transformations have a rich history in the study of Brownian motion and related processes. They provide elegant constructions of objects such as the Brownian bridge, meander, first-passage bridge, and excursion, while illuminating the deep connections between them (see, e.g., \cite{Vervaat1979, Bertoin1994, Bertoin2003}). In particular, a cornerstone of this theory is the relationship between Brownian motion and the $3$-dimensional Bessel process ${\rm BES}(3)$; see, e.g., \cite{Williams1970, Getoor1979, Biane1987, Pitman1996} and references therein. This work is particularly motivated by the construction of Pitman and Yor \cite[Theorem 15]{Pitman2003}, who showed that a ${\rm BES}(3)$ process can be obtained from a Brownian path by excising its excursions below the past maximum process that reach zero and concatenating the remaining excursions. This raises the question of whether a similar excision procedure applied to a Brownian bridge results can be related to $3$-dimensional Bessel bridge (or, equivalently, a normalized Brownian excursion). This work provides a detailed analysis of this relationship, establishing the formal path transformation and its distributional implications.

To formalize this, let $\mathbf{C}([0,\infty), \mathbb{R})$ be the space of all continuous functions from $[0,\infty)$ to $\mathbb{R}$, equipped with the topology of uniform convergence on compact subsets of $[0,\infty)$. Throughout the paper, we work on the canonical Wiener space $\Omega \coloneqq \mathbf{C}([0,\infty), \mathbb{R})$, and $\mathcal{F}$ stands for its Borel $\sigma$-algebra. We denote by $(B(t), t \geq 0)$ the canonical process on $\Omega$; that is, $B(t)(\omega) = \omega(t)$ for $\omega \in \Omega$. Let $\mathbb{P}$ be the Wiener measure, under which the coordinate process is a standard {\sl Brownian motion} started at $0$; we refer to \cite[\S 3 in Chapter I]{Yor1999} for background. All processes considered in this paper (e.g., Brownian bridge, Brownian meander, Brownian excursion, and first-passage bridge) are defined as measurable functionals of the same coordinate process $B$. Whenever a transformation is only meaningful on a subset of paths of full $\mathbb{P}$-measure, we extend it arbitrarily outside this set (for instance, by setting it equal to the zero function). Since these exceptional sets are Borel and negligible, all such transformations become Borel maps on the full path space, and all identities  are understood as holding $\mathbb{P}$-almost surely.

Let $(B^{\rm br}(t),  t \in [0,1])$ be a {\sl Brownian bridge}, i.e.,
$B^{\rm br}(t) = B(t)-tB(1)$, for $t \in [0,1]$. Let $\mu$ be the (a.s.\ unique) time at which $B^{\rm br}$ attains its maximum on $[0,1]$ (see, e.g., \cite[Lemma 0]{Vervaat1979}). Define the process $(M^{\rm br}(t),  t \in [0,1])$ by letting
\begin{align}
M^{\rm br}(t)  \coloneqq \begin{cases}
\max_{0 \leq s \leq t} B^{\rm br}(s), & \quad \text{for} \quad t \in [0, \mu], \\
\max_{t \leq s \leq 1} B^{\rm br}(s),  & \quad \text{for} \quad t \in [\mu, 1].
\end{cases}
\end{align}
\noindent Consider the excursions of $B^{\rm br}$ below the process $M^{\rm br}$, and excise the excursions that reach level $0$. Then, close the gaps by joining the remaining excursions together; see figure \ref{fig:excise-bridge}. Denote the resulting process by $(Y^{\rm br}(t), t \in [0,\tau^{\rm br}])$, where $\tau^{\rm br}$ is the total length of all non-excised excursions of $B^{\rm br}$ below $M^{\rm br}$ (i.e., those that do not reach level $0$). More precisely, for $t \in [0,1]$, define 
\begin{align}
G_{t}^{\rm br} \coloneqq \sup \{ s \in [0,t]: M^{\rm br}(s) - B^{\rm br}(s) = 0\} \quad \text{and} \quad  D_{t}^{\rm br} \coloneqq \inf \{ s \in (t, 1]: M^{\rm br}(s)-B^{\rm br}(s) = 0\},
\end{align}
\noindent with the convention $D_{1}^{\rm br} \coloneqq 1$. For $t \in [0,1]$, let
\begin{align}
R^{\rm br}(t) \coloneqq \max_{G_{t}^{\rm br} \leq s \leq D_{t}^{\rm br}} (M^{\rm br}(s) - B^{\rm br}(s)), \quad
U^{\rm br}(s) \coloneqq \int_{0}^{t} \mathbf{1}_{\{R^{\rm br}(s) < M^{\rm br}(s) \}} {\rm d} s, 
\end{align}
\noindent and for $s \in [0, U^{\rm br}(1)]$, $\alpha^{\rm br}(s) \coloneqq \inf\{t \in [0,1]: U^{\rm br}(t) > s \}$, with the convention $\alpha^{\rm br}(U^{\rm br}(1))\coloneqq 1$. Then, $\tau^{\rm br} \coloneqq U^{\rm br}(1)$ and 
\begin{align}
Y^{\rm br}(t) \coloneqq B^{\rm br}(\alpha^{\rm br}(t)), \quad \text{for} \quad t \in [0, \tau^{\rm br}].
\end{align}
\noindent Since $\mu \in (0,1)$ $\mathbb{P}$-almost surely, we have that $\tau^{\rm br} \in (0,1]$ $\mathbb{P}$-almost surely. Let $(\tilde{Y}^{\rm br}(t), t \in [0,1])$ be the process obtained by Brownian scaling $Y^{\rm br}$ to the time interval  $[0,1]$, that is,
\begin{align} \label{eq24}
\tilde{Y}^{\rm br}(t) \coloneqq (\tau^{\rm br})^{-1/2} Y^{\rm br}(\tau^{\rm br}t), \quad \text{for} \quad t \in [0,1]. 
\end{align}
\noindent For the well-definedness of $\tilde{Y}^{\rm br}$ and $\tau^{\rm br}$, we refer to Section \ref{Sec:Excisingbridges}, specifically Proposition \ref{Pro2} and Corollary \ref{corollary5}.
\begin{figure}[t]
  \centering
  \includegraphics[width=0.82\linewidth]{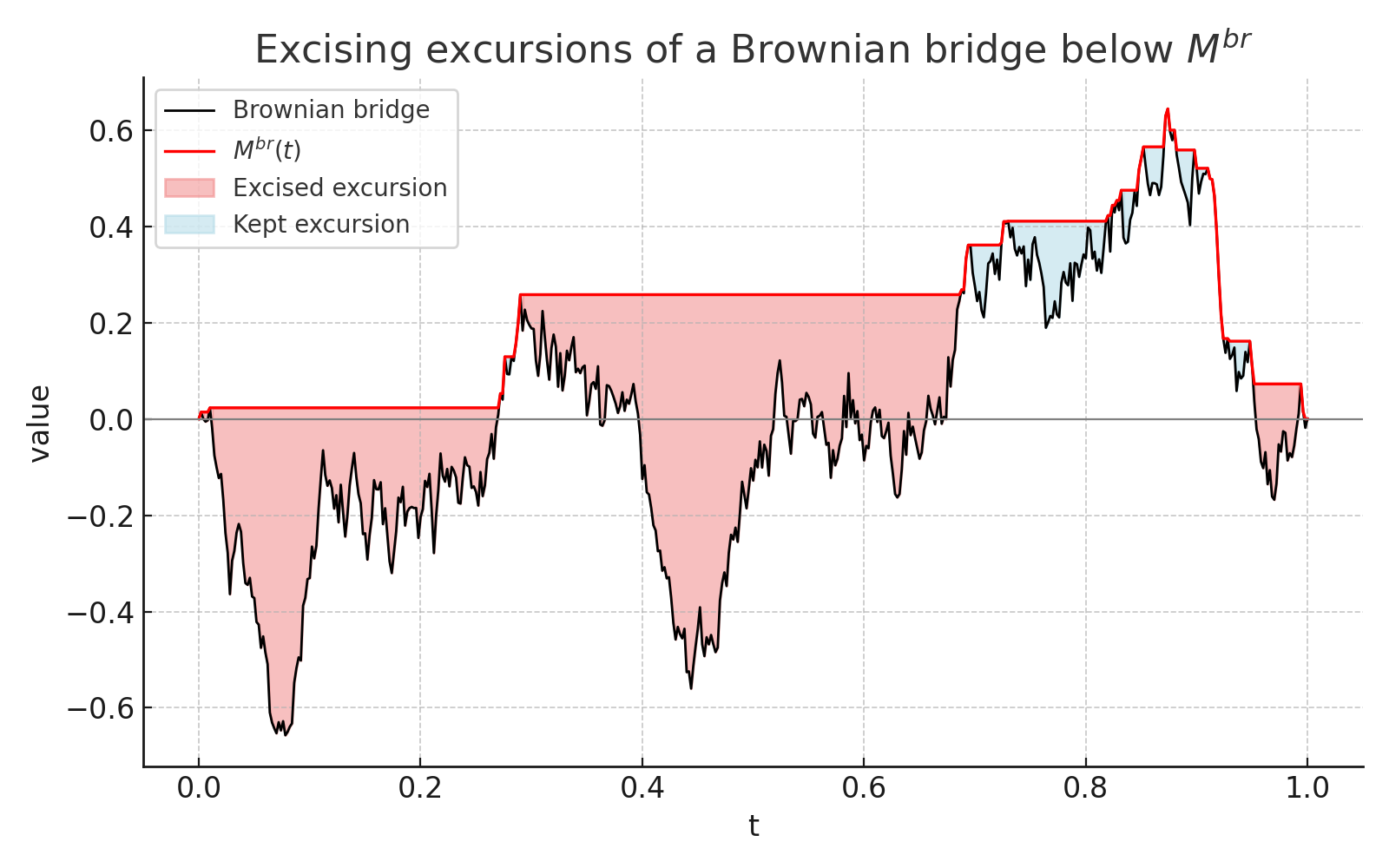}
  \caption{Excising excursions of a Brownian bridge below $M^{\mathrm{br}}(t)$.
  Excursions that reach level $0$ (red shading) are excised, while those that remain above $0$ (blue shading) are kept and concatenated.}
  \label{fig:excise-bridge}
\end{figure}

Our main result characterizes the joint law of $\tilde{Y}^{\rm br}$ and $\tau^{\rm br}$. Before we state our main result, we need to introduce some notation. For $x >0$, define
\begin{align} \label{eq36}
\phi_{x}(0) \coloneqq 0, \quad \text{and} \quad \phi_{x}(t) \coloneqq 2\int_{-\infty}^{1-(x/2t)^{2}} g_{x}(s) g_{x}(1-(x/2t)^{2}-s) {\rm d} s, \quad \text{for} \quad t >0,
\end{align}
\noindent where
\begin{align}  \label{eq15}
g_{x}(s) \coloneqq \frac{1}{x\sqrt{2\pi s}}(1-e^{-\frac{x^{2}}{2s}}) \mathbf{1}_{\{s\in (0,\infty)\}}.
\end{align}
\noindent As noted in Remark \ref{remark2} below, $g_{x}$ is a probability density function on $(0, \infty)$. Let $(B^{\rm exc}(t), t \in [0,1])$ be the {\sl (normalized) Brownian excursion} and $\bar{B}^{\rm exc}  \coloneqq \max_{0 \leq s \leq 1} B^{\rm exc}(s)$. Recall that $B^{\rm exc}$ can be constructed from $B^{\rm br}$ which is itself a function of $(B(t), t \geq 0)$; see, e.g., \cite{Vervaat1979} or \cite{Bertoin1994}. Recall also that $B^{\rm exc}$ has the same law as a $3$-dimensional Bessel bridge; see, e.g., \cite[Theorem 4.2 in Chapter XII]{Yor1999}. 

\begin{theorem} \label{Theorem1}
For all positive or bounded measurable function $G: \mathbf{C}([0,1], \mathbb{R}) \rightarrow \mathbb{R}$, we have that
\begin{align} \label{eq37}
\mathbb{E}[ G(\tilde{Y}^{\rm br}) (B^{\rm br}(\mu))^{-1}(\tau^{\rm br})^{1/2}] = \int_{0}^{\infty}  \mathbb{E}[G(B^{\rm exc})\phi_{x}(\bar{B}^{\rm exc})] {\rm d} x.
\end{align}
\end{theorem}

As a consequence of Theorem \ref{Theorem1}, we have the following corollary. 

\begin{corollary} \label{corollary2}
For all positive or bounded measurable function $G: \mathbb{R} \rightarrow \mathbb{R}$, we have that
\begin{align}
\mathbb{E}[ G(B^{\rm br}(\mu) (\tau^{\rm br})^{-1/2}) (B^{\rm br}(\mu))^{-1}(\tau^{\rm br})^{1/2}] = \int_{0}^{\infty}  \mathbb{E}[G(\bar{B}^{\rm exc})\phi_{x}(\bar{B}^{\rm exc})] {\rm d} x.
\end{align}
\end{corollary}

\begin{proof}
The result follows from Theorem \ref{Theorem1} by noting that $\max_{0 \leq t \leq 1} \tilde{Y}_{x}^{\rm br}(t) = B^{\rm br}(\mu)(\tau^{\rm br})^{-1/2}$.
\end{proof}

The remainder of the paper is organized as follows. Section \ref{sec: path} provides preliminaries on Brownian bridges, first passage bridges, and meanders, alongside the path transformations central to our construction. Section \ref{Sec:DeterministicExcising} focuses on the deterministic setting, formalizing the excision and concatenation transformations for bridge- and time-reversed meander-type functions and establishing their key properties. In Section \ref{sec:excise excursion}, we describe a process similar to $Y^{\rm br}$ by excising excursions of a standard Brownian motion below its past maximum process, followed by several essential lemmas. Finally, Section \ref{sec:proof-main} combines these results to provide the proof of Theorem \ref{Theorem1}.

\section{Some known Brownian path transformations} \label{sec: path}

In this section, we present several known Brownian path transformations that relate the Brownian bridge, the meander, and the first passage bridge.

The {\sl Brownian meander} $(B^{\rm me}(t), t \in [0,1])$ can be constructed from a Brownian bridge $(B^{\rm br}(t),  t \in [0,1])$ by setting (see, e.g., \cite[Theorem 2.2]{Bertoin1994})
\begin{align} \label{eq2}
B^{\rm me}(t)  = \begin{cases}
-B^{\rm br}(t+\mu)+B^{\rm br}(\mu), & \quad \text{for} \quad t \in [0, 1-\mu], \\
-B^{\rm br}(1-t) + 2 B^{\rm br}(\mu), & \quad \text{for} \quad   t \in [1-\mu,1],
\end{cases}
\end{align}
\noindent where $\mu$ is the (a.s.\ unique) instant when $B^{\rm br}$ attains its maximum valued on $[0,1]$. 

Define also the so-called {\sl time-space reversal of the Brownian meander} $(\overleftarrow{B}^{\rm me}(t), t \in [0,1])$ by letting 
\begin{align}  \label{eq3}
\overleftarrow{B}^{\rm me}(t) \coloneqq B^{\rm me}(1)- B^{\rm me}(1-t), \quad \text{for} \quad t \in [0,1].
\end{align}
\noindent Note that, $1-\mu = \sup\{ s \in [0,1]: B^{\rm me}(s) = \frac{1}{2} B^{\rm me}(1) \}$. Then, by \eqref{eq2},
\begin{align} \label{eq4}
\overleftarrow{B}^{\rm me}(t)  = \begin{cases}
B^{\rm br}(t), & \quad \text{for} \quad t \in [0, \mu], \\
B^{\rm br}(\mu) + B^{\rm br}(1+\mu-t), & \quad \text{for} \quad t \in [\mu, 1].
\end{cases}
\end{align}

\noindent In particular, $B^{\rm br}$ can be recovered from $\overleftarrow{B}^{\rm me}$ via the formulae 
\begin{align} \label{eq5}
B^{\rm br}(t)  = \begin{cases}
\overleftarrow{B}^{\rm me}(t), & \quad \text{for} \quad t \in [0, \mu], \\
\overleftarrow{B}^{\rm me}(1+\mu-t) -\overleftarrow{B}^{\rm me}(\mu),  & \quad \text{for} \quad t \in [\mu, 1].
\end{cases}
\end{align}
\noindent Compare with \cite[Theorem 2.2]{Bertoin1994}. (Note that \cite{Bertoin1994} uses the minimum instead of the maximum of the Brownian bridge.)

Let $T_{x} \coloneqq \inf \{t \geq 0: B(t)=x \}$ be the first hitting time of $x \in \mathbb{R}$ of the Brownian motion $(B(t), t \geq 0)$. The {\sl Brownian first passage bridge} $(F^{\rm br}_{x}(t),  t \in [0,1])$ of length $1$ from $0$ to $x>0$ is given by 
\begin{align} \label{eq19}
(F^{\rm br}_{x}(t), t \in [0,1]) \stackrel{d}{=} (B(t), t \in [0,1] \mid T_{x}=1),
\end{align}
\noindent where $\stackrel{d}{=}$ denotes equality in distribution, referring here to distribution on the space $\mathbf{C}([0,1], \mathbb{R})$. It is well-known that $(F^{\rm br}_{x}(t),  t \in [0,1])$ can also be constructed from the Brownian bridge $B^{\rm br}$, which is itself a functional of the standard Brownian motion
(see \cite[Theorem 7]{Bertoin2003}). We shall henceforth work with this version of
$(F^{\rm br}_{x}(t),  t \in [0,1])$. Furthermore, \cite[Corollary 8]{Bertoin2003} allows one to recover $B^{\rm br}$ from  $F^{\rm br}_{x}$. 

On the other hand, as pointed out in the paragraph after \cite[Proposition 5]{Bertoin2003}, it is known that for all positive or bounded measurable function $G: \mathbf{C}([0,1], \mathbb{R}) \rightarrow \mathbb{R}$, we have that
\begin{align} \label{eq1}
\mathbb{E}[G(F^{\rm br}_{x})] = \mathbb{E}[ G(\overleftarrow{B}^{\rm me}) \mid \overleftarrow{B}^{\rm me}(1) = x]. 
\end{align}

In particular, the next result shows that the law of the time-space reversal of a Brownian meander is a ``mixture'' of laws of Brownian first passage bridges. 
\begin{lemma} \label{Lemma2}
For all positive or bounded measurable function $G: \mathbf{C}([0,1], \mathbb{R}) \rightarrow \mathbb{R}$, we have that
\begin{align}
\int_{0}^{\infty} \mathbb{E}[G(F^{\rm br}_{x})] x e^{-\frac{x^{2}}{2}} {\rm d} x= \mathbb{E}[ G(\overleftarrow{B}^{\rm me})]. 
\end{align}
\end{lemma}

\begin{proof}
It is well-known (see, e.g., \cite{Bertoin1994})  that $B^{\rm me}(1) = 2 B^{\rm br}(\mu)$ is Rayleigh distributed (i.e., it has the distribution with probability density function $f(x) \coloneqq x e^{-\frac{x^{2}}{2}} \mathbf{1}_{\{ x \in (0, \infty) \}}$, for  $x \in \mathbb{R}$). Then, our claim follows from \eqref{eq1}, \eqref{eq2}, \eqref{eq3}. 
\end{proof}


\section{Deterministic path transformations} \label{Sec:DeterministicExcising}

In this section, we consider the deterministic setting. We begin in Section \ref{Sec:BridgetoMeander} by defining a transformation that maps bridges to time-reversed meander-type functions, along with its inverse. In Section \ref{Sec:Excisingbridges}, we formalize the process $\tilde{Y}^{\rm br}$ by defining the transformation that excises and concatenates excursions from bridges, as described in the introduction. Similarly, Section \ref{Sec:Excisingmeanders} defines the equivalent excision transformation for time-reversed meander-type functions. We also establish several continuity properties of these transformations that may be of independent interest. Finally, Section \ref{sec:excisecontinuousfunct} provides an analogous transformation for continuous functions before they reach a fixed level.

For a real number $a >0$, let $\mathbf{C}([0,a], \mathbb{R})$ be the space of real-valued continuous functions on $[0,a]$ equipped with the uniform topology.  For $w \in \mathbf{C}([0,a], \mathbb{R})$, let $\rho(w) \coloneqq \inf\{s \in [0,a]: w(s) = \max_{0 \leq u \leq a} w(u) \}$ be the first time at which the maximum of $w$ is attained. 

For $a>0$, let $\mathbf{B}_{a} \coloneqq \{ w \in \mathbf{C}([0,a], \mathbb{R}): w(0) = w(a)=0 \}$ be the subspace of bridges of length $a$ and let $\mathbf{B}^{\ast}_{a}$ be the subset of $\mathbf{B}_{a}$ such that each function has a unique location of its maximum. Also, let $\mathbf{M}_{a} \coloneqq \{w \in \mathbf{C}([0,a], \mathbb{R}): w(0)=0 \, \, \text{and} \, \, \max_{0 \leq s \leq a} = w(a)\}$ and let $\mathbf{M}_{a}^{\ast} \coloneqq \{w \in \mathbf{C}([0,a], \mathbb{R}): w(0)=0 \, \, \text{and} \, \, \rho(w) = a\}$.

\subsection{From bridges to time-reversed meander-type functions} \label{Sec:BridgetoMeander}

We begin by introducing a transformation that maps bridges to time-reversed meander-type functions. For $a>0$ fixed, define $\mathcal{T}^{\rm me}_{a}: \mathbf{B}_{a} \rightarrow \mathbf{C}([0,a], \mathbb{R})$ by letting $\mathcal{T}^{\rm me}_{a}(w) = (\overleftarrow{w}(t), t \in [0,a])$, for $w \in \mathbf{B}_{a}$, where
\begin{align} 
\overleftarrow{w}(t)  = \begin{cases}
w(t), & \quad \text{for} \quad t \in [0, \rho], \\
w(\rho) + w(a+\rho-t), & \quad \text{for} \quad t \in [\rho, a],
\end{cases}
\end{align}
\noindent and $\rho = \rho(w)$. 

Next, we define the inverse transformation of $\mathcal{T}^{\rm me}_{a}$. For $a>0$ fixed, define $\mathcal{T}^{\rm br}_{a}: \mathbf{M}_{a} \rightarrow \mathbf{C}([0,a], \mathbb{R})$ by letting $\mathcal{T}^{\rm br}_{a}(w) = (w^{\rm br}(t), t \in [0,a])$, for $w \in \mathbf{M}_{a}$, where
\begin{align} \label{eq49}
w^{\rm br}(t)  = \begin{cases}
w(t), & \quad \text{for} \quad t \in [0, \gamma_{w(a)/2}], \\
w(a+\gamma_{w(a)/2}-t) -w(\gamma_{w(a)/2}),  & \quad \text{for} \quad t \in [\gamma_{w(a)/2}, a],
\end{cases}
\end{align}
\noindent and $\gamma_{w(a)/2} \coloneqq \inf\{s \in [0,a]: w(s) = w(a)/2 \}$. 

If $a=1$, we write simply  $\mathcal{T}^{\rm me}$ and $\mathcal{T}^{\rm br}$ instead of $\mathcal{T}^{\rm me}_{1}$ and $\mathcal{T}^{\rm br}_{1}$, respectively.

\begin{remark} \label{Remark9}
If $w \in \mathbf{B}_{a}$, then $\mathcal{T}^{\rm me}_{a}(w) \in \mathbf{M}_{a}$. Furthermore, if $w \in \mathbf{B}_{a}^{\ast}$, then $\mathcal{T}^{\rm me}_{a}(w) \in \mathbf{M}_{a}^{\ast}$. In this case, the unique location of the maximum of $\mathcal{T}^{\rm me}_{a}(w)$ is $a$, $\overleftarrow{w}(a) = 2 w(\rho)$, $\gamma_{\overleftarrow{w}(a) /2} = \rho$, and $\mathcal{T}^{\rm br}_{a} \circ \mathcal{T}^{\rm me}_{a}(w) = w$. 

Similarly, if $w \in \mathbf{M}_{a}$, then $\mathcal{T}^{\rm br}_{a}(w) \in \mathbf{B}_{a}$. Furthermore, if $w \in \mathbf{M}_{a}^{\ast}$, then $\mathcal{T}^{\rm br}_{a}(w) \in \mathbf{B}_{a}^{\ast}$. In this case, the unique location of the maximum of $\mathcal{T}^{\rm br}_{a}(w)$ is $\gamma_{w(a)/2}$, $w^{\rm br}(\gamma_{w(a)/2}) = w(a)/2$ and $\mathcal{T}^{\rm me}_{a} \circ \mathcal{T}^{\rm br}_{a}(w) =w$. 
\end{remark}

For notational convenience, we also use $\mathcal{T}^{\rm me}_{a}$ and $\mathcal{T}^{\rm br}_{a}$ to denote their restrictions to the domains $\mathbf{B}^{\ast}_{a}$ and $\mathbf{M}^{\ast}_{a}$, respectively, so that $\mathcal{T}^{\rm me}_{a}: \mathbf{B}^{\ast}_{a} \rightarrow \mathbf{C}([0,a], \mathbb{R})$ and $\mathcal{T}^{\rm br}_{a}: \mathbf{M}_{a}^{\ast} \rightarrow \mathbf{C}([0,a], \mathbb{R})$. 

\begin{proposition} \label{Pro3}
For $a>0$ fixed, the maps $\mathcal{T}^{\rm me}_{a}: \mathbf{B}^{\ast}_{a} \rightarrow \mathbf{C}([0,a], \mathbb{R})$ and $\mathcal{T}^{\rm br}_{a}: \mathbf{M}_{a}^{\ast} \rightarrow \mathbf{C}([0,a], \mathbb{R})$ are continuous with respect the uniform topology and (Borel) measurable.
\end{proposition}

As we could not find an exact reference for Proposition \ref{Pro3}, we provide a proof in Appendix \ref{AProofofPro3} for the sake of completeness, though the result may be known.

\begin{remark} \label{Remark2}
Since $B^{\rm br}$ attains its maximum on $[0,1]$ at a unique location $\mathbb{P}$-almost surely (see, e.g., \cite[Lemma 0]{Vervaat1979}), we have that $\mathbb{P}(B^{\rm br} \in \mathbf{B}_{1}^{\ast})=1$. Furthermore, it follows, for instance, from \eqref{eq4} that $\mathbb{P}(\overleftarrow{B}^{\rm me} \in \mathbf{M}_{1}^{\ast})=1$. Note also from \eqref{eq4} and \eqref{eq5} that $\overleftarrow{B}^{\rm me} = \mathcal{T}^{\rm me}(B^{\rm br})$ and $B^{\rm br}=\mathcal{T}^{\rm br}(\overleftarrow{B}^{\rm me})$, respectively.
\end{remark}

We conclude this section by defining a transformation that we name the Brownian scaling transformation. For $a>0$ fixed, define $\mathcal{S}_{a}: \mathbf{C}([0,a], \mathbb{R}) \rightarrow \mathbf{C}([0,1], \mathbb{R})$ by letting $\mathcal{S}_{a}(w) = (\tilde{w}(t), t \in [0,1])$, for $w \in \mathbf{C}([0,a], \mathbb{R})$, where $\tilde{w}(t) = a^{-1/2}w(at)$, for $t \in [0,1]$. If $a=1$, we write simply  $\mathcal{S}$ instead of $\mathcal{S}_{1}$. The following result is straightforward and its proof is omitted.

\begin{lemma}
For $a>0$ fixed, $\mathcal{S}_{a}$ is continuous with respect to the uniform topology and (Borel) measurable.
\end{lemma}

\subsection{Excising excursions from bridges} \label{Sec:Excisingbridges}

In this section, we define a transformation that excises excursions below the maximum of a bridge whenever they reach level $0$. For $w \in \mathbf{B}_{1}$, let $\rho = \rho(w)$ and define
\begin{align}
m^{\rm br}_{t}(w)  \coloneqq \begin{cases}
\max_{0 \leq s \leq t} w(s), & \quad \text{for} \quad t \in [0, \rho], \\
\max_{t \leq s \leq 1} w(s),  & \quad \text{for} \quad t \in [\rho, 1].
\end{cases}
\end{align}
\noindent Define also $g_{t}^{\rm br}(w)$ and $d_{t}^{\rm br}(w)$ the left and right endpoints of the excursion of $m^{\rm br}(w)- w$ that straddle time $t \in [0, 1]$, that is,
\begin{align} \label{eq42}
g_{t}^{\rm br}(w) \coloneqq \sup \{ s \in [0,t]: m^{\rm br}_{s}(w) - w(s) = 0\} \quad
\text{and} \quad 
d_{t}^{\rm br}(w)  \coloneqq \inf \{ s \in (t, 1]: m^{\rm br}_{s}(w)-w(s) = 0\},
\end{align}
\noindent with the convention 
$d_{1}^{\rm br}(w)  \coloneqq 1$. Let
\begin{align}
r^{\rm br}_{t}(w)  \coloneqq \max_{g_{t}^{\rm br}(w) \leq s \leq d_{t}^{\rm br}(w)} (m^{\rm br}_{s}(w) - w(s)), \quad \text{for} \quad t \in [0,1],
\end{align}
\noindent be the maximum of $m^{\rm br}(w)- w$ over the excursion interval $[g_{t}^{\rm br}(w), d_{t}^{\rm br}(w)]$, 
\begin{align} \label{eq38}
u_{t}^{\rm br}(w) \coloneqq \int_{0}^{t} \mathbf{1}_{\{r^{\rm br}_{s}(w) < m^{\rm br}_{s}(w) \}} {\rm d} s, \quad \text{for} \quad t \in [0,1],
\end{align}
\noindent and
\begin{align} \label{eq41}
\alpha_{s}^{\rm br}(w) = \alpha_{s}^{\rm br} \coloneqq \inf\{t \in [0,1]: u_{t}^{\rm br}(w) > s \}, \quad \text{for} \quad s \in [0, u_{1}^{\rm br}(w)],
\end{align}
\noindent with the convention $\alpha_{u_{1}^{\rm br}(w)}^{\rm br}(w) \coloneqq 1$. Then, define 
\begin{align}
\mathcal{H}^{\rm br}(w)\coloneqq ( w(\alpha_{t}^{\rm br}), t \in [0,u_{1}^{\rm br}(w)]).
\end{align}

Observe that $\mathcal{H}^{\rm br}(w)$ is the function obtained by excising the excursions of $w$ below $m^{\rm br}$ that reach level $0$ and then closing the gaps by joining the remaining excursions together. In particular, $w(\alpha_{t}^{\rm br}) \geq 0$ for all $t \in [0,u_{1}^{\rm br}(w)]$, with $w(\alpha_{0}^{\rm br}) = w(\alpha_{u_{1}^{\rm br}(w)}^{\rm br}) = 0$. Thus, if $u_{1}^{\rm br}(w) >0$, then $\mathcal{H}^{\rm br}(w)\in \mathbf{B}_{u_{1}^{\rm br}(w)}$. 
Note also that $u_{t}^{\rm br}(w)$ is the length of all non-excised excursions up to time $t$. In particular, $u_{1}^{\rm br}(w)$ is the total length of all non-excised excursions. 

Suppose that $w \in \mathbf{B}_{1}^{\ast}$. Then, $\rho = \rho(w)$ is the unique location of the of the maximum of $w$. In particular, $\rho \in (0,1)$ and thus,
$u_{1}^{\rm br}(w) \in (0,1]$. Furthermore,  $u_{\rho}^{\rm br}(w)$ is the unique location of the maximum of $\mathcal{H}^{\rm br}(w)$, $w(u_{\rho}^{\rm br}(w)) = w(\rho)$ and $\mathcal{H}^{\rm br}(w) \in \mathbf{B}^{\ast}_{u_{1}^{\rm br}(w)}$.

Define $\mathcal{G}^{\rm br}:  \mathbf{B}_1 \rightarrow \mathbf{C}([0,1], \mathbb{R})$ by letting $\mathcal{G}^{\rm br}(w) \coloneqq \mathcal{S}_{u_{1}^{\rm br}(w)} \circ \mathcal{H}^{\rm br}(w)$, for $w \in \mathbf{B}_1$, if $u_{1}^{\rm br}(w) >0$, and setting $\mathcal{G}^{\rm br}(w)$ equal to the constant zero function on $[0,1]$ otherwise. The main result of this section shows that $\mathcal{G}^{\rm br}$ is continuous when restricted to a specific subset of functions in $\mathbf{B}_1$.  

\begin{definition}[Regularity] \label{DefinitionRegular}
We say that $w\in \mathbf{B}_1$ is \emph{regular} if it satisfies the following two properties:
\begin{enumerate}[label=(\textbf{\roman*})]
\item \label{RegPro1} \textbf{Uniqueness of one-sided maximizers at rational times.} For every rational $t \in [0, 1]$, there exists a unique $\vec{\rho}_{t} \in [0, t]$ such that $\max_{0 \leq s \leq t}w(s) = w(\vec{\rho}_{t})$. Furthermore, there also exists a unique $\protect\cev{\rho}_{t} \in [t, 1]$ such that $\max_{t \leq s \leq 1}w(s)   = w(\protect\cev{\rho}_{t})$.

\item \label{RegPro2} \textbf{No strict local minimum at $0$.}  The level $0$ is not a strict local minimum of $w$.
\end{enumerate}

\noindent We denote by $\mathbf{B}^{{\rm reg}}_{1}$ the subset of $\mathbf{B}_1$ consisting of regular functions.
\end{definition}

With a slight abuse of notation, we let $\mathcal{G}^{\rm br}$ denote its restriction to $\mathbf{B}^{{\rm reg}}_1$, that is, $\mathcal{G}^{\rm br}: \mathbf{B}^{{\rm reg}}_1 \rightarrow \mathbf{C}([0,1], \mathbb{R})$. The proof of the following result is deferred to Appendix \ref{AProofofPro2}.

\begin{proposition} \label{Pro2}
The map $\mathcal{G}^{\rm br}: \mathbf{B}^{{\rm reg}}_1 \rightarrow \mathbf{C}([0,1], \mathbb{R})$ is continuous with respect to the uniform topology and (Borel) measurable.
\end{proposition}

\begin{corollary} \label{corollary5} 
We have that $\mathbb{P}(B^{\rm br}\in \mathbf{B}^{\rm reg}_1)=1$. 
\end{corollary}

\begin{proof}
Recall that $(B(t), t \geq 0)$ denotes a standard Brownian motion started at $B(0)=0$. It is well-known that, for every $a \in [0,1]$, $(B(t), t \in [0,a])$ attains its absolute maximum at a unique time with probability one; see, e.g., \cite[Theorem 2.11]{Moters2010} or \cite[Exercise (3.26) in Chapter III]{Yor1999}. Similarly, one can deduce that $(B(t), t \in [a,1])$ attains its absolute maximum at a unique time with probability $1$. Recall that the distribution of $(B(t), t \in [0,1])$ and $(B^{\rm br}(t), t \in [0,1])$ are equivalent; see \cite[Section 3 in Chapter VIII]{Bertoin1996} or \cite[Exercise (3.16) in Chapter I]{Yor1999}. Therefore, it follows that with probability one $B^{\rm br}$ satisfies property \ref{RegPro1} in Definition \ref{DefinitionRegular}; cf. also with  \cite[Proof of Lemma 0]{Vervaat1979}. 

On the other hand, it is not difficult to prove that $0$ is not a strict local minimum of $(B(t), t \in [0,1])$ (cf. with \cite[Exercise (3.26) in Chapter III]{Yor1999}). Therefore it follows that with probability one $B^{\rm br}$ satisfies property \ref{RegPro2} in Definition \ref{DefinitionRegular}. This concludes our proof. 
\end{proof}

\begin{remark} \label{Remark1}
Note that $Y^{\rm br} = \mathcal{H}^{\rm br}(B^{\rm br})$ and $\tau^{\rm br}=u_{1}^{\rm br}(B^{\rm br})$. Furthermore, $\mathcal{G}^{\rm br}(B^{\rm br}) = \tilde{Y}^{\rm br}$. 
\end{remark}

\subsection{Excising excursions from time-reversed meander-type functions} \label{Sec:Excisingmeanders}

In this section, we define the transformation that excises the excursions below the maximum of time-reversed meander-type functions. For $w \in \mathbf{M}_{1}$, let $\bar{w}(t)  \coloneqq \max_{0 \leq s \leq t} w(s)$, for $t \in [0, 1]$. Define also $g_{t}^{\rm me}(w)$ and $d_{t}^{\rm me}(w)$ the left and right endpoints of the excursion of $\bar{w}- w$ that straddle time $t \in [0, 1]$, that is,
\begin{align}
g_{t}^{\rm me}(w) \coloneqq \sup \{ s \in [0,t]: \bar{w}(s) - w(s) = 0\} 
\quad \text{and} \quad d_{t}^{\rm me}(w) \coloneqq \inf \{ s \in (t, 1]: \bar{w}(s)-w(s) = 0\},
\end{align}
\noindent with the convention 
$d_{1}^{\rm me}(w) \coloneqq 1$. Let
\begin{align}
r_{t}^{\rm me}(w)  \coloneqq \max_{g_{t}^{\rm me}(w) \leq s \leq d_{t}^{\rm me}(w)} (\bar{w}(s) - w(s)), \quad \text{for} \quad t \in [0,1],
\end{align}
\noindent be the maximum of $\bar{w} - w$ over the excursion interval $[g_{t}^{\rm me}(w), d_{t}^{\rm me}(w)]$. Recall that $\gamma_{w(1)/2} \coloneqq \inf\{s\in[0,1]:w(s)=w(1)/2\}$ and define
\begin{align} \label{eq47}
u_{t}^{\rm me}(w) \coloneqq \begin{cases}
\int_{0}^{t} \mathbf{1}_{\{r_{s}^{\rm me}(w) < \bar{w}(s) \}} {\rm d} s, & \quad \text{for} \quad t \in [0, \gamma_{w(1)/2}], \\
\int_{0}^{\gamma_{w(1)/2}} \mathbf{1}_{\{r_{s}^{\rm me}(w) < \bar{w}(s) \}} {\rm d} s + \int_{\gamma_{w(1)/2}}^{t} \mathbf{1}_{\{r_{s}^{\rm me}(w) < \bar{w}(s) -w(1)/2 \}} {\rm d} s,  & \quad \text{for} \quad t \in [\gamma_{w(1)/2}, 1],
\end{cases}
\end{align}
\noindent and
\begin{align}
\alpha_{s}^{\rm me}(w) = \alpha_{s}^{\rm me} \coloneqq \inf\{t \in [0,1]: u_{t}^{\rm me}(w) > s \}, \quad \text{for} \quad s \in [0, u_{1}^{\rm me}(w)].
\end{align}
\noindent with the convention $\alpha_{u_{1}^{\rm me}(w)}^{\rm me}(w) \coloneqq 1$. Then, let
\begin{align} \label{eq48}
\mathcal{H}^{\rm me}(w)\coloneqq ( w(\alpha_{t}^{\rm me}), t \in [0,u_{1}^{\rm me}(w)]).
\end{align}

Suppose that $w \in \mathbf{M}_{1}^{\ast}$. Then,  $w(1)>0$, $\gamma_{w(1)/2} \in (0,1)$ and $u_{1}^{\rm me}(w) \in (0,1]$. Note also that, $w(\alpha_{u_{1}^{\rm me}(w)}^{\rm me}) = w(1)$ and $\mathcal{H}^{\rm me}(w) \in \mathbf{M}_{u_{1}^{\rm me}(w)}^{\ast}$. 

Define $\mathcal{G}^{\rm me}:  \mathbf{M}_1  \rightarrow \mathbf{C}([0,1], \mathbb{R})$ by letting $\mathcal{G}^{\rm me}(w) \coloneqq \mathcal{S}_{u_{1}^{\rm me}(w)} \circ \mathcal{H}^{\rm me}(w)$, for $w \in \mathbf{M}_1$, if  $u_{1}^{\rm me}(w) > 0$, and setting $\mathcal{G}^{\rm me}(w)$ equal to the constant zero function on $[0,1]$ otherwise. The following properties, whose proofs we omit, follow directly from the definitions of $\mathcal{T}^{\rm br}_{a}$, $\mathcal{T}^{\rm me}_{a}$, $\mathcal{S}_{a}$, $\mathcal{H}^{\rm br}$, $\mathcal{H}^{\rm me}$, $\mathcal{G}^{\rm br}$ and $\mathcal{G}^{\rm me}$. 

\begin{lemma} \label{lemma8}
For $w \in \mathbf{B}_{1}$, $\mathcal{T}^{\rm br} \circ \mathcal{G}^{\rm me}\circ \mathcal{T}^{\rm me}(w) = \mathcal{G}^{\rm br}(w)$. Furthermore, $u_{1}^{\rm br}(w) = u_{1}^{\rm me}(\mathcal{T}^{\rm me}(w))$.
\end{lemma}

\begin{lemma} \label{lemma9}
For $w \in \mathbf{M}_{1}$, $\mathcal{T}^{\rm br} \circ \mathcal{G}^{\rm me}(w) = \mathcal{S}_{u_{1}^{\rm me}(w)}  \circ \mathcal{T}^{\rm br}_{u_{1}^{\rm me}(w)} \circ \mathcal{H}^{\rm me}(w)$ and $\mathcal{T}^{\rm me}  \circ \mathcal{G}^{\rm br} \circ \mathcal{T}^{\rm br}(w) = \mathcal{G}^{\rm me}(w)$. Furthermore, $u_{1}^{\rm me}(w) = u_{1}^{\rm br}(\mathcal{T}^{\rm br}(w))$.
\end{lemma}

We will also show that $\mathcal{G}^{\rm me}$ is continuous when restricted to a specific subset of functions in $\mathbf{M}_1$.  

\begin{definition}[Regularity] \label{DefinitionRegularII}
We say that $w \in \mathbf{M}_1$ is \emph{regular} if $\mathcal{T}^{\rm br}(w) \in \mathbf{B}_{1}^{{\rm reg}}$ (i.e., $\mathcal{T}^{\rm br}(w)$ satisfies \ref{RegPro1} and \ref{RegPro2} in Definition \ref{DefinitionRegular}). Let $\mathbf{M}^{{\rm reg}}_{1}$ be the subset of $\mathbf{M}_1$ consisting of regular functions.
\end{definition}

In particular, for  $w \in \mathbf{M}_1^{\rm reg}$, $\mathcal{G}^{\rm me}(w)$ is strictly positive on $(0,1]$ and is zero only at $0$. With a slight abuse of notation, we let $\mathcal{G}^{\rm me}$ denote its restriction to $\mathbf{M}^{\rm reg}_{1}$, that is,  $\mathcal{G}^{\rm me}: \mathbf{M}^{\rm reg}_{1} \rightarrow \mathbf{C}([0,1], \mathbb{R})$. The proof of the following result is provided in Appendix \ref{AProofofPro4}. 

\begin{proposition} \label{Pro4}
The map $\mathcal{G}^{\rm me}: \mathbf{M}^{\rm reg}_{1} \to  \mathbf{C}([0,1], \mathbb{R})$ is continuous with respect to the uniform topology and (Borel) measurable.
\end{proposition}

\begin{remark} \label{Remark3}
From Remark \ref{Remark2} and Corollary \ref{corollary5}, we have that $\mathbb{P}(\overleftarrow{B}^{\rm me} \in \mathbf{M}^{\rm reg}_{1})=1$. Furthermore, by Lemma \ref{lemma8}, $\mathcal{G}^{\rm br}(B^{\rm br}) = \mathcal{T}^{\rm br} \circ \mathcal{G}^{\rm me}(\overleftarrow{B}^{\rm me})$. 
\end{remark}

\begin{corollary} \label{corollary1}
The map $\mathcal{T}^{\rm br} \circ \mathcal{G}^{\rm me}: \mathbf{M}^{\rm reg}_{1} \to  \mathbf{C}([0,1], \mathbb{R})$ is continuous with respect to the uniform topology and (Borel) measurable.
\end{corollary}

\begin{proof}
It follows from Propositions \ref{Pro3} and \ref{Pro4}. 
\end{proof}

\begin{remark} \label{Remark5}
Recall that $F^{\rm br}_{x}$ denotes the Brownian first passage bridge of length $1$ from $0$ to $x>0$, defined as a measurable functional of the canonical Brownian motion. Then $\mathbb{P}(F_x^{\rm br}\in \mathbf M_1^\ast)=1$. In particular, $\mathbb{P}$-almost surely: $F_x^{\rm br}(0)=0$, $F_x^{\rm br}(1)=x$, the path is continuous,
and $F_x^{\rm br}(t)<x$ for all $t\in[0,1)$. Furthermore, by~\eqref{eq1} and Remark~\ref{Remark3}, it follows that $\mathbb{P}(F_x^{\rm br}\in \mathbf M_1^{\rm reg})=1$. 
\end{remark}

\subsection{Excising excursions from continuous functions} \label{sec:excisecontinuousfunct}

In this section, we define a transformation similar to $\mathcal{H}^{\rm me}$ and $\mathcal{G}^{\rm me}$ that excises excursions from functions in  $\mathbf{C}([0,\infty), \mathbb{R})$ below their past maximum. 

For $w \in \mathbf{C}([0,\infty), \mathbb{R})$, let $\bar{w}(t)  \coloneqq \max_{0 \leq s \leq t} w(s)$, for $t \geq 0$. For $y > 0$, let $\gamma_{y} \coloneqq \inf\{s\geq 0: w(s)=y\}$. In this section, $\inf \emptyset = \infty$. Define also 
\begin{align}
g_{t}(w) \coloneqq \sup \{ s \leq t: \bar{w}(s) - w(s) = 0\} \quad \text{and}
\quad  d_{t}(w) \coloneqq \inf \{ s >t: \bar{w}(s)-w(s) = 0\},
\end{align}

For $a>0$, let $\mathbf{C}_{a} \coloneqq \{ w \in \mathbf{C}([0,\infty), \mathbb{R}): w(0) = 0 \, \, \text{and} \, \, \gamma_{a} < \infty\}$. Then, for $w \in \mathbf{C}_{a}$, define
\begin{align}
r_{t}(w)  \coloneqq 
\max_{g_{t}(w) \leq s \leq d_{t}(w)} (\bar{w}(s) - w(s)),  \quad \text{for} \quad t \in [0, \gamma_{a}],
\end{align}
\begin{align} \label{eq45}
u_{t}(w) \coloneqq \begin{cases}
\int_{0}^{t} \mathbf{1}_{\{r_{s}(w) < \bar{w}(s) \}} {\rm d} s, & \quad \text{for} \quad t \in [0, \gamma_{a/2}], \\
\int_{0}^{\gamma_{a/2}} \mathbf{1}_{\{r_{s}(w) < \bar{w}(s) \}} {\rm d} s + \int_{\gamma_{a/2}}^{t} \mathbf{1}_{\{r_{s}(w) < \bar{w}(s) -a/2 \}} {\rm d} s,  & \quad \text{for} \quad t \in [\gamma_{a/2}, \gamma_{a}],
\end{cases}
\end{align}
\noindent and
\begin{align}
\alpha_{s}(w) = \alpha_{s} \coloneqq \inf\{t \in [0, \gamma_{a}]: u_{t}(w) > s \}, & \quad \text{for} \quad s \in [0, u_{\gamma_{a}}(w)],
\end{align}
\noindent with the convention $\alpha_{u_{\gamma_{a}}(w)}(w) = \alpha_{u_{\gamma_{a}}(w)} \coloneqq \gamma_{a}$. Finally, define
\begin{align} \label{eq46}
\mathcal{H}_{a}(w)\coloneqq ( w(\alpha_{t}), t \in [0,u_{\gamma_{a}}(w)]).
\end{align}

Observe that since $w \in \mathbf{C}_{a}$, we have $u_{\gamma_{a}}(w) \in (0, \gamma_{a}]$, $w(\alpha_{u_{\gamma_{a}}(w)}) = w(\gamma_{a}) = a>0$ and $\mathcal{H}_{a}(w) \in \mathbf{M}_{u_{\gamma_{a}}(w)}$. For $a>0$ fixed, define $\mathcal{G}_{a}:  \mathbf{C}_{a}  \rightarrow \mathbf{C}([0,1], \mathbb{R})$ by letting $\mathcal{G}_{a}(w) \coloneqq \mathcal{S}_{u_{\gamma_{a}}(w)} \circ \mathcal{H}_{a}(w)$, for $w \in \mathbf{C}_{a}$. If $a=1$, we write simply $\mathcal{G}$ instead of $\mathcal{G}_{1}$. The following property follow directly from our definitions and its proof is omitted.

\begin{lemma} \label{lemma11}
For $w \in \mathbf{C}_{a}$, $\mathcal{T}^{\rm br}_{u_{\gamma_{a}}(w)} \circ \mathcal{G}_{a}(w) = \mathcal{S}_{u_{\gamma_{a}}(w)}  \circ \mathcal{T}^{\rm br}_{u_{\gamma_{a}}(w)} \circ \mathcal{H}_{a}(w)$. 
\end{lemma}

\begin{remark} \label{Remark8}
For $w \in \mathbf{C}_{a}$ with $\gamma_{a} = 1$, let $w\vert_{[0,1]}\coloneqq (w(t), t\in[0,1])$ denote the restriction of $w$ to the interval $[0,1]$. In particular, $w\vert_{[0,1]}\in \mathbf{M}_{1}$, $\mathcal{H}_{1}(w) = \mathcal{H}^{\rm me}(w\vert_{[0,1]})$, $\mathcal{G}(w) = \mathcal{G}^{\rm me}(w\vert_{[0,1]})$ and $u_{1}(w) = u_{1}^{\rm me}(w\vert_{[0,1]})$. 
\end{remark}

\begin{remark} \label{Remark6}
For every $a >0$, $\mathbb{P}(B \in \mathbf{C}_{a} )=1$, where $B$ is the canonical Brownian motion. 
\end{remark}

\section{Excising Excursions from the standard Brownian motion} \label{sec:excise excursion} 

In this section, we describe the excursion structure of $(B(t), t \in [0, T_{x}])$, for $x > 0$, and define a process, similar to $Y^{\rm br}$, by excising the excursions of $(B(t), t \in [0, T_{x}])$ below its past maximum process. Recall that $T_{x} \coloneqq \inf \{t \geq 0: B(t)=x \}$.

Let $(\bar{B}(t), t \geq 0)$ be the past maximum process of $B$, i.e., $\bar{B}(t) \coloneqq \max_{0 \leq s \leq t} B(s)$, for $t \geq 0$. The excursions of $\bar{B}-B$ away from $0$ correspond to those of $B$ below its continuously increasing past maximum process $\bar{B}$. Specifically, each excursion of $\bar{B}-B$ away from $0$ is associated with a unique level $\ell$, which is the constant value of $\bar{B}$ throughout the duration of the excursion. The duration of this excursion is given by $\zeta_{\ell}=T_{\ell}-T_{\ell-}$. Moreover, an excursion at level $\ell$ has height equals to $h_{\ell} = \ell-\min_{T_{\ell-} \leq s \leq T_{\ell}} B(s)$.

Fix $x > 0$ and excise some of the excursions of $(B(t), t \in [0, T_{x}])$ below its past maximum, according to the following two rules: 
\begin{enumerate}[label=(\textbf{R.\arabic*})]
  \item\label{rule1} An excursion at level $\ell \in (0,x/2)$ of $(B(t), t \in [0, T_{x/2}])$ below its past maximum is excised if and only if its height $h_{\ell}$ satisfies $h_{\ell} \geq \ell$ (i.e., the associated piece of trajectory of $(B(t), t \in [0, T_{x/2}])$ hits zero during such an excursion interval). 
  \item\label{rule2} An excursion of at level $\ell \in [x/2,x)$ of $(B(t), t \in [T_{x/2}, T_{x}])$ is excised if and only if its height $h_{\ell}$ satisfies $h_{\ell} \geq \ell-x/2$ (i.e., the associated piece of trajectory of $(B(t)-x/2, t \in [T_{x/2}, T_{x}])$ hits zero during such an excursion interval). 
\end{enumerate}

Let $\tau_{x}$ be the total length of the non-excised excursions of $(B(t), t \in [0, T_{x}])$, and let $(X_{x}(t), t \in [0, \tau_{x}])$ be the resulting process obtained by joining the non-excised excursions together. Specifically, 
\begin{align} \label{eq43}
\tau_{x} = u_{T_{x}}(B)  \quad \text{and} \quad (X_{x}(t), t \in [0, \tau_{x}]) = \mathcal{H}_{x}(B),
\end{align}
\noindent where $u_{T_{x}}(\cdot)$ and $\mathcal{H}_{x}(\cdot)$ are defined in \eqref{eq45} and \eqref{eq46}, respectively. In particular, $\tau_{x} \in (0, T_{x}]$ $\mathbb{P}$-almost surely. (Recall Remark \ref{Remark6}.)

\begin{remark}
It is worth noting that Pitman and Yor \cite[Section 7.4]{Pitman2003} considered the excision of excursions of $(B(t), t \geq 0)$ below its past maximum whose height $h_{\ell}$ satisfies $h_{\ell} \geq \ell$. 
\end{remark}

On the other hand, the L\'evy-It\^{o} excursion theory (see, e.g., \cite[Chapter XII]{Yor1999}) allows us to decompose a Brownian path into a Poisson point process (PPP) of excursions and then reconstruct the original path from them, yielding another construction of $X_{x}$ and $\tau_{x}$. Recall that $\mathbf{C}([0,\infty), \mathbb{R})$ denotes the space of all continuous functions from $[0,\infty)$ to $\mathbb{R}$, equipped with the topology of uniform convergence on every compact subset of $[0,\infty)$. For $w \in \mathbf{C}([0,\infty), \mathbb{R})$, we set $\zeta(w) \coloneqq \inf\{t >0: w(t)=0 \}$ (where $\sup \emptyset = 0$ by convention) and
\begin{align}
\mathcal{E} \coloneqq \{ w \in \mathbf{C}([0,\infty), \mathbb{R}): w(0) = 0, \, 0 < \zeta(w) < \infty \, \, \text{and} \, \, w(t)=0 \, \, \text{for all} \, \, t \geq \zeta(w)\}. 
\end{align}
\noindent Let $\partial$ be the function that is identically $0$ and set $\mathcal{E}_{\partial} = \mathcal{E} \cup \{ \partial\}$. It is well-known that there is a PPP on $(0, \infty) \times \mathcal{E}_{\partial}$,
\begin{align}
\mathcal{N}(\cdot) \coloneqq \sum_\ell \delta_{((\ell,e_\ell) \in \cdot)},
\end{align}
\noindent where the sum is over the random countable set of $\ell$ with $T_{\ell}-T_{\ell-}>0$, with intensity measure ${\rm d}t \otimes \mathbf{n}({\rm d} e)$, where $\mathbf{n}({\rm d} e)$ is a $\sigma$-finite measure on $\mathcal{E}_{\partial}$ known as the It\^{o} excursion measure. 

Note that, for a fixed $x >0$, the structure of the excursions of $(B(t), t \in [0, T_{x}])$ is described by the PPP $\mathcal{N}$ restricted to $(0,x) \times \mathcal{E}_{\partial}$. In particular, thinning $\mathcal{N}$ according to rules \ref{rule1} and \ref{rule2} yields independent PPPs, $\mathcal{N}_{1,\mathrm{n}}$ and $\mathcal{N}_{2,\mathrm{n}}$ on $(0,x/2) \times \mathcal{E}_{\partial}$ and $[x/2,x) \times \mathcal{E}_{\partial}$, respectively. These are the PPPs of non-excised excursions of $(B(t), t \in [0, T_{x}])$. By the Poisson reconstruction principle, concatenating these excursions in chronological order recovers the path $(X_x(t), t \in [0,\tau_x])$; cf.\ \cite[Proposition 2.5 in Chapter XII]{Yor1999}. 
Furthermore, $\tau_x$ can also be recover from $\mathcal{N}$, this is discussed in detail in the next section, see \eqref{eq17}. 

\subsection{The law of $\tau_{x}$} \label{sec:lengthofnon-excisedexc}

We consider the PPP derived from the heights and lengths of Brownian excursions below its past maximum process. Specifically, let $N$ be the PPP obtained as the image of $\mathcal{N}$ under the mapping $(\ell,e)\mapsto(\ell,\zeta(e),\bar{e})$, where $\bar{e} \coloneqq \sup_{0 \leq s \leq \zeta(e)} e(s)$. It is well-known (see, e.g., \cite{Biane1987} or \cite[Proposition 12]{Pitman2003}) that 
\begin{align} \label{eq33}
N(\cdot) \coloneqq \sum_{\ell} \delta_{((\ell, \zeta_{\ell}, h_{\ell})\in \cdot)},
\end{align}
\noindent where the sum is over the random countable set of $\ell$ with $T_{\ell}-T_{\ell-}>0$, is a PPP on $(0, \infty) \times (0, \infty) \times (0, \infty)$ with intensity measure ${\rm d} \ell \mu({\rm d} v {\rm d} m)$ where $(v, m)$ denotes a generic value of $(\zeta_{\ell}, h_{\ell})$,
\begin{equation}
\mu({\rm d} v {\rm d} m) =\frac{{\rm d}  v}{\sqrt{2 \pi} v^{3 / 2}} \mathbb{P}\left(\sqrt{v} \bar{B}^{\rm exc} \in {\rm d} m\right),
\end{equation}
\noindent $(B^{\rm exc}(t), t \in [0,1])$ is the (normalized) Brownian excursion and $\bar{B}^{\rm exc}  = \max_{0 \leq s \leq 1} B^{\rm exc}(s)$. In particular, for a fixed $x >0$, the structure of the heights and lengths of the excursions of $(B(t), t \in [0, T_{x}])$ is described by the PPP $N$ restricted to $(0,x) \times (0, \infty) \times (0, \infty)$.

Define 
\begin{align}
N_{1, {\rm n}}(\cdot) \coloneqq \sum_{\ell} \delta_{((\ell, \zeta_{\ell}, h_{\ell})\in \cdot)} \mathbf{1}_{\{ \ell < x/2, h_{\ell} < \ell \}} \quad \text{and} \quad N_{2, {\rm n}}(\cdot) \coloneqq \sum_{\ell} \delta_{((\ell, \zeta_{\ell}, h_{\ell})\in \cdot)} \mathbf{1}_{\{ x/2 \leq \ell < x , h_{\ell} < \ell-x/2 \}}.
\end{align}
\noindent 
Clearly, $N_{1, {\rm n}}$ and $N_{2, {\rm n}}$ are two independent PPPs on $(0, \infty) \times (0, \infty) \times (0, \infty)$ with intensity measures 
\begin{align}
{\rm d} \ell \mu({\rm d} v {\rm d} m) \mathbf{1}_{\{ \ell < x/2, m < \ell \}}  \quad \text{and} \quad {\rm d} \ell \mu({\rm d} v {\rm d} m) \mathbf{1}_{\{ x/2 \leq \ell < x , m < \ell-x/2 \}},
\end{align}
\noindent respectively. The heights and lengths of the non-excised excursions of $(B(t), t \in [0, T_{x/2}])$ are described by the PPP $N_{1, {\rm n}}$, while the heights and lengths of the non-excised excursions of $(B(t), t \in [T_{x/2}, T_{x}])$ are described by the PPP $N_{2, {\rm n}}$. In particular, the total length of the non-excised excursions $\tau_{x}$ according to rules \ref{rule1}-\ref{rule2} satisfies
\begin{align} \label{eq17}
\tau_{x} = \int_{(0,\infty)} s N_{\rm n}({\rm d}s),
\end{align}
\noindent where
\begin{align} \label{eq6}
N_{\rm n}(\cdot) \coloneqq N_{1, {\rm n}}((0, \infty) \times \cdot \times (0, \infty)) +  N_{2, {\rm n}}((0, \infty) \times \cdot \times (0, \infty))
\end{align}
\noindent is a PPP on $(0, \infty)$ with intensity measure $\nu_{\rm n}$ given by
\begin{align}
\nu_{\rm n}({\rm d} s)  = \frac{\sqrt{2}}{ \sqrt{ \pi s}} \int_{0}^{\frac{x}{2 \sqrt{s}}}   \mathbb{P}(\bar{B}^{\rm exc} \leq m ) {\rm d} m {\rm d} s.
\end{align}

Let $(R_{3}(t), t \geq 0)$ and  $(\hat{R}_{3}(t), t \geq 0)$ be two independent ${\rm BES}(3)$ processes started from $R_{3}(0)=\hat{R}_{3}(0)=0$, defined in some probability space $(\Omega^{\prime}, \mathcal{F}^{\prime}, \mathbf{P})$. Let $H_{y} \coloneqq \inf\{t \geq 0: R_{3}(t) = y \}$ and $\hat{H}_{y} \coloneqq \inf\{t \geq 0: \hat{R}_{3}(t) = y \}$ be the first hitting times of $y \geq 0$ of $R_{3}$ and $\hat{R}_{3}$, respectively. 

\begin{lemma} \label{lemma3}
For $x > 0$, we have that $\tau_{x} \stackrel{d}{=} H_{x/2} + \hat{H}_{x/2}$.
\end{lemma}

\begin{proof}
It follows from \cite[Corollary 14]{Pitman2003} that the heights and lengths of the excursions of $(R_{3}, t \in [0, H_{x/2}])$ below its past maximum are described by the PPP $N_{1, {\rm n}}$. Then,  the total length of the non-excised excursions of $(B(t), t \in [0, T_{x/2}])$ is equal in distribution to 
$H_{x/2}$. On the other hand, one can deduce (for, e.g., by using \eqref{eq6}) that the total length of the non-excised excursions of $(B(t), t \in [T_{x/2}, T_{x}])$ is equal in distribution to $\hat{H}_{x/2}$. The above implies our claim. 
\end{proof}

It is well-known (see, e.g., \cite{KentJohn1978} or \cite[(3.12)]{Pitman1996}) that, for $y >0$ and $\lambda \geq 0$,
\begin{align} \label{eq31}
\mathbf{E}[e^{-\lambda H_{y}}] = \frac{y \sqrt{2 \lambda}}{\sinh(y\sqrt{2 \lambda})}.
\end{align}
\noindent In particular, it follows from Lemma \ref{lemma3} that, for $x >0$ and $\lambda \geq 0$,
\begin{align} \label{eq16}
\mathbb{E}[e^{-\lambda \tau_{x}}] =  \left(\frac{(x \sqrt{\lambda})/\sqrt{2}}{ \sinh((x \sqrt{\lambda})/\sqrt{2})} \right)^{2}.
\end{align}
\noindent Moreover, the probability density function $f_{\tau_{x}}$ of $\tau_{x}$ is given by 
\begin{align} \label{eq32}
f_{\tau_{x}}(t) = \int_{-\infty}^{t} f_{x/2}^{(3)}(s) f_{x/2}^{(3)}(t-s) {\rm d} s, \quad \text{for} \quad t \in \mathbb{R},
\end{align}
\noindent where $f_{x/2}^{(3)}$ is the probability density function of $H_{x/2}$. 

Let us now consider the total length $\tau_{x}^{\rm e}$ of the excised excursions of $(B(t), t \in [0, T_{x}])$. It follows from \ref{rule1} and \ref{rule2}, and a similar argument to the one made before, that
\begin{align} \label{eq44}
\tau_{x}^{\rm e}= \int_{(0,\infty)} s N_{\rm e}({\rm d}s),
\end{align}
\noindent where
\begin{align} 
N_{\rm e}(\cdot) \coloneqq \sum_{\ell} \delta_{((\ell, \zeta_{\ell}, h_{\ell})\in (0,\infty) \times \cdot \times (0, \infty))} \left( \mathbf{1}_{\{ \ell < x/2, h_{\ell} \geq \ell \}} + \mathbf{1}_{\{ x/2 \leq \ell < x , h_{\ell} \geq \ell-x/2 \}}\right)
\end{align}
\noindent is a PPP on $(0, \infty)$ with intensity measure $\nu_{\rm e}$ given by
\begin{align}
\nu_{\rm e}({\rm d} s)  = \frac{\sqrt{2}}{ \sqrt{ \pi s}} \int_{0}^{\frac{x}{2 \sqrt{s}}} (1- \mathbb{P}(\bar{B}^{\rm exc} \leq m )) {\rm d} m {\rm d} s.
\end{align}

\begin{lemma} \label{lemma2}
For $x > 0$ and $\lambda \geq 0$, we have that
\begin{align} \label{eq7}
\mathbb{E}[e^{-\lambda \tau_{x}^{\rm e}}] = \left( \frac{1-e^{-x\sqrt{2\lambda}}}{x\sqrt{2\lambda}}  \right)^{2}. 
\end{align}
\noindent Moreover, the probability density function $f_{\tau_{x}^{\rm e}}$ of $\tau_{x}^{\rm e}$ is given by 
\begin{align} \label{eq8}
f_{\tau_{x}^{\rm e}}(t) = \int_{-\infty}^{t} g_{x}(s) g_{x}(t-s) {\rm d} s, \quad \text{for} \quad t \in \mathbb{R}, 
\end{align}
\noindent where the function $g_{x}$ is defined in \eqref{eq15}.
\end{lemma}

\begin{proof}
Observe that, for $x > 0$, $\tau_{x}^{\rm e} + \tau_{x} = T_{x}$. In particular (see, e.g., \cite[Proposition 3.7 in Chapter II]{Yor1999}), we have that, for $\lambda \geq 0$,
\begin{align} \label{eq9}
\mathbb{E}[e^{-\lambda T_{x}}] = e^{-x\sqrt{2\lambda}}.
\end{align} 
\noindent On the other hand, recall that $\sinh(u) = 2^{-1}(1-e^{-2u})e^{-u}$, for $u \in \mathbb{R}$. Then, by \eqref{eq16}, we deduce that
\begin{align} \label{eq10}
\mathbb{E}[e^{-\lambda \tau_{x}}] = \left( \frac{x \sqrt{2\lambda}}{1-e^{-x\sqrt{2\lambda}}} \right)^{2} e^{-x\sqrt{2\lambda}}. 
\end{align}
\noindent Note that $\tau_{x}^{\rm e}$ and $\tau_{x}$ are independent. Then, we necessarily must have that
\begin{align} \label{eq11}
\mathbb{E}[e^{-\lambda T_{x}}] = \mathbb{E}[e^{-\lambda \tau_{x}}]\mathbb{E}[e^{-\lambda \tau_{x}^{\rm e}}].
\end{align} 
\noindent Thus, \eqref{eq7} follows from \eqref{eq9}, \eqref{eq10} and \eqref{eq11}. 

Note that, for $\lambda>0$,
\begin{align} \label{eq12}
\frac{1}{x\sqrt{2\lambda}} = \int_{0}^{\infty} \frac{1}{x\sqrt{2\pi s}}e^{-\lambda s} {\rm d} s.
\end{align}
\noindent Moreover, 
\begin{align} \label{eq13}
\frac{e^{-x\sqrt{2\lambda}}}{x\sqrt{2\lambda}} = \int_{0}^{\infty} \frac{1}{x\sqrt{2\pi s}}e^{-\frac{x^{2}}{2s}}e^{-\lambda s} {\rm d} s;
\end{align}
\noindent see, e.g., \cite[(3.3)]{Getoor1979}. Then, it follows from \eqref{eq12} and \eqref{eq13} that,
\begin{align} \label{eq14}
\frac{1-e^{-x\sqrt{2\lambda}}}{x\sqrt{2\lambda}} = \int_{0}^{\infty} \frac{1}{x\sqrt{2\pi s}}(1-e^{-\frac{x^{2}}{2s}})e^{-\lambda s} {\rm d} s.
\end{align}
\noindent Therefore, \eqref{eq8} follows from \eqref{eq15}, \eqref{eq7}, \eqref{eq14} and the properties of the Laplace transform. 
\end{proof}

\begin{remark} \label{remark2}
The function $g_x$ is a probability density on $(0,\infty)$. Indeed, by \eqref{eq15} and \eqref{eq14}, the Laplace transform $\mathcal{L}\{g_x\}$ of $g_{x}$ is
\begin{align} 
\mathcal{L}\{g_x\}(\lambda) = \int_0^\infty e^{-\lambda s} g_x(s) {\rm d} s = \frac{1-e^{-x\sqrt{2\lambda}}}{x\sqrt{2\lambda}}, \quad \text{for} \quad \lambda >0. 
\end{align}
\noindent By taking $\lambda\downarrow0$, we observe that $\mathcal{L}\{g_x\}(0+)=1$. Hence $\int_0^\infty g_x(s) {\rm d} s=1$. Moreover, $g_x(s) \sim (x\sqrt{2\pi s})^{-1}$ as $s \downarrow 0$, while $g_x(s) \sim x (2 s^{3/2}\sqrt{2\pi})^{-1}$ as $s \to \infty$. Thus, $g_x$ is integrable at both endpoints.
\end{remark}

\begin{remark} \label{remark1}
The ${\rm BES}(3)$ processes $(R_{3}(t), t \geq 0)$ started from $R_{3}(0)=0$ inherits the Brownian scaling property, that is,  for every $c>0$, $(c^{-1/2}R_{3}(ct), t \geq 0) \stackrel{d}{=} (R_{3}(t), t \geq 0)$.  In particular, for $y >0$, $H_{y} \stackrel{d}{=} y^{2} H_{1}$. Thus, by Lemma \ref{lemma3}, $\tau_{x} \stackrel{d}{=} (x/2)^{2} \tau_{1}$. 

On the other hand, observe from \eqref{eq15} that, for $x > 0$ and $s \in \mathbb{R}$, $g_{x}(s) = x^{-2} g_{1}(x^{-2}s)$.
\noindent Thus, it follows from Lemma \ref{lemma2} that $\tau_{x}^{{\rm e}} \stackrel{d}{=} x^{2} \tau_{1}^{{\rm e}}$ and in particular, $f_{\tau_{x}^{{\rm e}}}(s) = x^{-2} f_{\tau_{1}^{{\rm e}}}(x^{-2}s)$.
\end{remark}

\subsection{The law of $X_{x}$}

Let $(R_{3}(t), t \geq 0)$ and  $(\hat{R}_{3}(t), t \geq 0)$ be two independent ${\rm BES}(3)$ processes started from $R_{3}(0)=\hat{R}_{3}(0)=0$, defined in some probability space $(\Omega^{\prime}, \mathcal{F}^{\prime}, \mathbf{P})$. Recall that  $H_{y}$ and $\hat{H}_{y}$ denote the first hitting times of level $y \geq 0$ for $R_{3}$ and $\hat{R}_{3}$, respectively. Fix $x > 0$ and define
\begin{align} \label{eq34}
W(t)  \coloneqq \begin{cases}
R_{3}(t), & \quad \text{for} \quad   t \in [0, H_{x/2}], \\
\frac{x}{2}+\hat{R}_{3}(t-H_{x/2}), & \quad \text{for} \quad  t \in [H_{x/2}, H_{x/2} + \hat{H}_{x/2}].
\end{cases}
\end{align}

\begin{lemma} \label{lemma1}
For $x > 0$, $(X_{x}(t), t \in [0, \tau_{x}]) \stackrel{d}{=} (W(t), t \in [0, H_{x/2} + \hat{H}_{x/2}])$.
\end{lemma}

\begin{proof}
Recall that in Section \ref{sec:excise excursion}, $(X_x(t), t \in [0,\tau_x])$ can be constructed using the PPP $\mathcal{N}$ of excursions of $(B_{t}, t \geq 0)$ below its past maximum process. Recall also that $\mathcal{N}_{1,\mathrm{n}}$ and $\mathcal{N}_{2,\mathrm{n}}$ denote the two independent PPPs on $(0,x/2) \times \mathcal{E}_{\partial}$ and $[x/2,x) \times \mathcal{E}_{\partial}$, respectively, obtained by thinning $\mathcal{N}$ according to rules \ref{rule1} and \ref{rule2}, respectively.

It is not difficult to see that the PPP of excursions of the $\mathrm{BES}(3)$  process $R_{3}$ below its past maximum process, restricted to levels $\ell\in (0,x/2)$, has the same law as $\mathcal{N}_{1,\mathrm{n}}$; see, e.g., \cite[Corollary 14 and its proof]{Pitman2003}. Similarly, the PPP of excursions of $\hat{R}_{3}(t)$ below its past maximum process, restricted to levels $\ell\in [x/2,x)$, has the same law as $\mathcal{N}_{2,\mathrm{n}}$. Therefore, the concatenation of non-excised excursions of $(B(t), t \in [0, T_{x}])$ up to level $x/2$ has the same law as $(R_3(t), t \in [0, H_{x/2}])$, while the concatenation of non-excised excursions from level $x/2$ up to level $x$ has the law of $(x/2+ \hat{R}_3(t),\,t \in [0, \hat{H}_{x/2}])$, independent of $(R_3(t), t \in [0, H_{x/2}])$. Therefore, the path $(X_x(t), t \in [0,\tau_x])$ has the same law as the path $W$ in \eqref{eq34}. This proves the lemma.
\end{proof}

Clearly, $X_{x}(0) = 0$, $X_{x}(t) < x$, for $t \in [0, \tau_{x})$ and $X_{x}(\tau_{x}) =x$. Moreover, $X_{x}$ has continuous trajectories. Define the process $(X_{x}^{\rm br}(t), t \in [0,1])$ by transforming $X_{x}$ into a bridge from $0$ to $0$, that is, 
\begin{align} \label{eq25}
X_{x}^{\rm br}(t)  \coloneqq \mathcal{T}^{\rm br}(X_{x}),
\end{align}
\noindent where $\mathcal{T}^{\rm br}(\cdot)$ is defined in \eqref{eq49}. Let $(\tilde{X}_{x}^{\rm br}(t), t \in [0,1])$ be the process  obtained by Brownian scaling $X_{x}^{\rm br}$ to the time interval $[0,1]$, that is,
\begin{align} \label{eq28}
\tilde{X}_{x}^{\rm br}(t) \coloneqq \tau_{x}^{-1/2} X_{x}^{\rm br}(\tau_{x} t), \quad \text{for} \quad t \in [0,1]. 
\end{align}
\noindent In particular, by \eqref{eq43} and Lemma \ref{lemma11}, $\tilde{X}_{x}^{\rm br} = \mathcal{S}_{\tau_{x}}  \circ \mathcal{T}^{\rm br}_{\tau_{x}} \circ \mathcal{H}_{x}(B)$.

\begin{lemma}[An agreement formula] \label{lemma4}
For all positive or bounded measurable function $G: \mathbf{C}([0,1], \mathbb{R}) \rightarrow \mathbb{R}$, we have that
\begin{align}
\mathbb{E}[G(B^{\rm exc})] = \frac{\sqrt{2 \pi}}{x} \mathbb{E}[G(\tilde{X}_{x}^{\rm br}) \tau_{x}^{1/2}].
\end{align}
\end{lemma}

\begin{proof}
Define 
\begin{align}
W^{\rm br}(t)  \coloneqq \begin{cases}
W(t), & \quad \text{for} \quad   t \in [0, H_{x/2}], \\
W(2H_{x/2} + \hat{H}_{x/2}-t) - \frac{x}{2}, & \quad \text{for} \quad t \in [H_{x/2}, H_{x/2} + \hat{H}_{x/2}].
\end{cases}
\end{align}
\noindent Let $(\tilde{W}^{\rm br}(t), t \in [0,1])$ be the process  obtained by Brownian scaling $W^{\rm br}$ to the time interval $[0,1]$, that is,
\begin{align}
\tilde{W}^{\rm br}(t) = (H_{x/2} + \hat{H}_{x/2})^{-1/2} W^{\rm br}((H_{x/2} + \hat{H}_{x/2}) t), \quad \text{for} \quad t \in [0,1]. 
\end{align}
\noindent Recall that $B^{\rm exc}$ has the same law as a $3$-dimensional Bessel bridge; see, e.g., \cite[Theorem 4.2 in Chapter XII]{Yor1999}. Then, by an argument similar to the proof of \cite[Theorem 3.1]{Pitman1996} mutatis-mutandis, it can be shown that, for all positive or bounded measurable function $G: \mathbf{C}([0,1], \mathbb{R}) \rightarrow \mathbb{R}$,
\begin{align}
\mathbb{E}[G(B^{\rm exc})] = \sqrt{\frac{\pi}{2}} \mathbf{E}\left[G(\tilde{W}^{\rm br}) \Big(\max_{0 \leq t \leq 1}\tilde{W}^{\rm br}(t) \Big)^{-1} \right] = \frac{\sqrt{2 \pi}}{x} \mathbf{E}[G(\tilde{W}^{\rm br}) (H_{x/2} + \hat{H}_{x/2})^{1/2}].
\end{align}
\noindent To obtain the second identity, we used that $\max_{0\le t\le 1}\tilde W^{\rm br}(t)=\tfrac{x}{2}(H_{x/2}+\hat H_{x/2})^{-1/2}$. Therefore, our claim follows from Lemma \ref{lemma1}. 
\end{proof}

\section{Proof of the main result}\label{sec:proof-main}

In this section we prove Theorem~\ref{Theorem1}. 

\begin{proof}[Proof of Theorem \ref{Theorem1}]
Recall that $(Y^{\rm br}(t), t \in [0,\tau^{\rm br}])$ denotes the process resulting from excising the excursions of $B^{\rm br}$ below $M^{\rm br}$ that reach level $0$ and joining the remaining non-excised excursions together, and  $\tau^{\rm br}$ is the total length of all non-excised excursions of $B^{\rm br}$. Recall also that $(\tilde{Y}^{\rm br}(t), t \in [0,1])$ denotes the process obtained by Brownian scaling $Y^{\rm br}$ to the time interval $[0,1]$; see \eqref{eq24}.

The argument proceeds in three steps:
\begin{enumerate}[label=(\textbf{\roman*})]
\item \label{Step1} Using \eqref{eq4}, \eqref{eq5} and Lemma \ref{lemma8}, we link the excision and concatenation procedure of the Brownian bridge's excursions to the equivalent procedure for the time-space reversal Brownian meander; see \eqref{eq18}.

\item \label{Step2} Then, we establish the weighted expectation identity for the law of $\tilde{Y}^{\rm br}$ by conditioning on the Brownian bridge maximum (using \eqref{eq22}) and the maximum-scaling relationship \eqref{eq29} and \eqref{eq35}. This establishes a link to an equivalent excision and concatenation procedure for the Brownian first-passage bridge \eqref{eq:A2}.  

\item \label{Step3} Finally, we conclude the proof by first linking the Brownian first-passage bridge's excision and concatenation procedure to that of $(B(t), t \in [0, T_{x}])$ described in Section \ref{sec:excise excursion}, and then applying Lemma \ref{lemma4} to identify the kernel $\phi_{x}$ in \eqref{eq36}. 
\end{enumerate}

\paragraph{\ref{Step1} From the bridge to the meander.} Observe that, by \eqref{eq4} and \eqref{eq5}, the excursions of $(B^{\rm br}(t), t \in [0, \mu])$ below $(M^{\rm br}(t), t \in [0, \mu])$ are exactly those of the time-space reversal of the Brownian meander $(\overleftarrow{B}^{\rm me}(t), t \in [0, \mu])$ below its past maximum process. On the other hand, the excursions of $(B^{\rm br}(t), t \in [\mu, 1])$ below $(M^{\rm br}(t), t \in [\mu, 1])$ correspond to those of $(B^{\rm br}(1-t), t \in [0, 1-\mu])$ below its past maximum process which in turn correspond to those of $(\overleftarrow{B}^{\rm me}(t), t \in [\mu,1])$ below its past maximum process. In particular, the excised excursions of $(B^{\rm br}(t), t \in [0, \mu])$ are those of $(\overleftarrow{B}^{\rm me}(t), t \in [0, \mu])$ below its past maximum process that reach level $0$, while the excised  excursions of $(B^{\rm br}(t), t \in [\mu, 1])$ are those of $(\overleftarrow{B}^{\rm me}(t), t \in [\mu, 1])$ below its past maximum process that reach level $\overleftarrow{B}^{\rm me}(\mu)$. 

Recall the definition of the transformation $\mathcal{H}^{\rm me}(\cdot)$ in \eqref{eq48}, which excises excursions below the past maximum of the time-reversed meander-type functions. Recall also that $u_{1}^{\rm me}(\cdot)$ in \eqref{eq47} is the function that returns the length of the non-excised excursions after applying $\mathcal{H}^{\rm me}$. Then, let $\tau^{\rm me} \coloneqq u_{1}^{\rm me}(\overleftarrow{B}^{\rm me})$ be the total length of the non-excised excursions of $\overleftarrow{B}^{\rm me}$ and $(Z^{\rm me}(t), t \in [0,\tau^{\rm me}]) \coloneqq \mathcal{H}^{\rm me}(\overleftarrow{B}^{\rm me})$ be the resulting process obtained by joining the non-excised excursions together. Recall from Remark \ref{Remark3} that $\mathbb{P}(\overleftarrow{B}^{\rm me} \in \mathbf{M}_{1}^{\rm reg})=1$; in particular, $\tau^{\rm me} \in (0,1]$ $\mathbb{P}$-almost surely. Let $(\hat{Z}^{\rm me}(t), t \in [0,1])$ be the process obtained by Brownian scaling $Z^{\rm me}$ to the time interval $[0,1]$, that is,
\begin{align} 
\hat{Z}^{\rm me}(t) \coloneqq (\tau^{\rm me})^{-1/2} Z^{\rm me}(\tau^{\rm me}t), \quad \text{for} \quad t \in [0,1]. 
\end{align}
\noindent In particular, $\hat{Z}^{\rm me} = \mathcal{G}^{\rm me}(\overleftarrow{B}^{\rm me})$, where $\mathcal{G}^{\rm me}(\cdot)$ is the transformation defined in Section \ref{Sec:Excisingmeanders}. Furthermore,  recall from Lemma \ref{lemma8} and Remark \ref{Remark3}  that
\begin{align} \label{eq18}
\tilde{Y}^{\rm br} = \mathcal{T}^{\rm br}(\hat{Z}^{\rm me}) \quad \text{and} \quad \tau^{\rm br}=\tau^{\rm me},
\end{align}
\noindent where $\mathcal{T}^{\rm br}(\cdot)$ is the transformation defined in \eqref{eq49}

\paragraph{\ref{Step2} Weighted expectation via conditioning.} By \eqref{eq3} and \eqref{eq1}, conditional on $\overleftarrow{B}^{\rm me}(1)  =x$ (recall that, by \eqref{eq4}, $2B^{\rm br}(\mu) = \overleftarrow{B}^{\rm me}(1)$), for $x > 0$, the excised excursions of $(B^{\rm br}(t), t \in [0, \mu])$ are those of $(F_{x}^{\rm br}(t), t \in [0, \gamma_{x/2}^{\rm br}])$ below its past maximum process that reach level $0$, while the excised excursions of $(B^{\rm br}(t), t \in [\mu, 1])$ are those of $(F_{x}^{\rm br}(t), t \in [\gamma_{x/2}^{\rm br}, 1])$ below its past maximum process that reach level $x/2$, where $\gamma_{x/2}^{\rm br} \coloneqq \inf\{t \in [0,1]:F_{x}^{\rm br}(t) = x/2 \}$. 

Recall from Remark \ref{Remark5} that $\mathbb{P}(F_{x}^{\rm br} \in \mathbf{M}_{1}^{\rm reg})=1$. Let $\tau_{x}^{\rm br} \coloneqq u_{1}^{\rm me}(F_{x}^{\rm br})$ be the total length of the non-excised excursions of $F_{x}^{\rm br}$, as described above, and let $(Z_{x}(t), t \in [0, \tau_{x}^{\rm br}]) \coloneqq \mathcal{H}^{\rm me}(F_{x}^{\rm br})$ be the process obtained by joining the non-excised excursions together. ($\tau_{x}^{\rm br} \in (0,1]$ $\mathbb{P}$-almost surely.) Define the process $(Z_{x}^{\rm br}(t), t \in [0,1])$ by transforming  $Z_{x}$ into a bridge from $0$ to $0$, that is, 
\begin{align} \label{eq27}
Z^{\rm br}_{x}(t) = \begin{cases}
Z_{x}(t), & \quad \text{for} \quad t \in [0, \gamma_{x/2}^{\rm br}], \\
Z_{x}(\gamma_{x/2}^{\rm br}+\tau_{x}^{\rm br}-t) - x/2, & \quad \text{for} \quad t \in [\gamma_{x/2}^{\rm br}, \tau_{x}^{\rm br}].
\end{cases}
\end{align}
\noindent Note that $(Z_{x}^{\rm br}(t), t \in [0,1]) \coloneqq \mathcal{T}^{\rm br}_{u_{1}^{\rm me}(F_{x}^{\rm br})}(Z_{x})$. Let $(\tilde{Z}_{x}^{\rm br}(t), t \in [0,1])$ be the process obtained by Brownian scaling $Z_{x}^{\rm br}$ to the time interval $[0,1]$, that is,
\begin{align} \label{eq26}
\tilde{Z}_{x}^{\rm br}(t) \coloneqq (\tau^{\rm br}_{x})^{-1/2} Z_{x}^{\rm br}(\tau^{\rm br}_{x}t), \quad \text{for} \quad t \in [0,1]. 
\end{align}
\noindent Indeed, by Lemma \ref{lemma9}, $(\tilde{Z}_{x}^{\rm br}(t), t \in [0,1]) = \mathcal{T}^{\rm br} \circ \mathcal{G}^{\rm me}(F_{x}^{\rm br})$. Then, it follows from \eqref{eq1}, \eqref{eq18} and Corollary \ref{corollary1} that, for all positive or bounded measurable function $G^{\prime}: \mathbf{C}([0,1], \mathbb{R}) \rightarrow \mathbb{R}$,
\begin{align}
\mathbb{E}[G^{\prime}(\tilde{Z}_{x}^{\rm br})] & =  \mathbb{E}[ G^{\prime}(\mathcal{T}^{\rm br} \circ \mathcal{G}^{\rm me}(F_{x}^{\rm br}) ) ] \nonumber \\
& =  \mathbb{E}[ G^{\prime}(\mathcal{T}^{\rm br} \circ \mathcal{G}^{\rm me}(\overleftarrow{B}^{\rm me}) ) \mid \overleftarrow{B}^{\rm me}(1) =x] \nonumber \\
& = \mathbb{E}[ G^{\prime}(\tilde{Y}^{\rm br}) \mid 2B^{\rm br}(\mu) =x]. 
\end{align}
\noindent Furthermore, by Lemma \ref{Lemma2}, for all positive or bounded measurable function $G^{\prime}: \mathbf{C}([0,1], \mathbb{R}) \rightarrow \mathbb{R}$,
\begin{align} \label{eq22}
\int_{0}^{\infty} \mathbb{E}[G^{\prime}(\tilde{Z}_{x}^{\rm br})] x e^{-\frac{x^{2}}{2}} {\rm d} x= \mathbb{E}[ G^{\prime}(\tilde{Y}^{\rm br})]. 
\end{align}
\noindent By applying \eqref{eq22} with $G^{\prime}(f)=G(f)\,(\max_{0\le t\le1}f(t))^{-1}$, where $f \in \mathbf{C}([0,1], \mathbb{R})$ and $G: \mathbf{C}([0,1], \mathbb{R}) \rightarrow \mathbb{R}$ is a positive or bounded measurable function, we obtain that
\begin{align}\label{eq:A1}
\mathbb{E} \left[ G(\tilde{Y}^{\rm br}) \Big(\max_{0 \leq t \leq 1} \tilde{Y}^{\rm br}(t) \Big)^{-1}\right]
=\int_0^\infty \mathbb{E} \left[G(\tilde Z^{\rm br}_x) \Big(\max_{0 \leq t \leq 1}  \tilde{Z}^{\rm br}_x(t) \Big)^{-1}\right]  x e^{-\frac{x^{2}}{2}}{\rm d}x.
\end{align}

Next, observe from \eqref{eq27} and \eqref{eq26} that,
\begin{align} \label{eq29}
\max_{0 \leq t \leq 1} \tilde{Z}_{x}^{\rm br}(t) = \frac{x}{2}(\tau^{\rm br}_{x})^{-1/2},
\end{align}
\noindent while from \eqref{eq24}, we deduce that,
\begin{align} \label{eq35}
\max_{0\le t\le1}\tilde Y^{\rm br}(t)&=B^{\rm br}(\mu)(\tau^{\rm br})^{-1/2}.
\end{align}
\noindent Then, by combining \eqref{eq:A1}, \eqref{eq29} and \eqref{eq35}, we obtain that 
\begin{align} \label{eq:A2}
\mathbb{E} \left[ G(\tilde{Y}^{\rm br}) (B^{\rm br}(\mu))^{-1} (\tau^{\rm br})^{1/2} \right]
= 2 \int_0^\infty \mathbb{E} \left[G(\tilde Z^{\rm br}_x) (\tau^{\rm br}_{x})^{1/2}\right]  e^{-\frac{x^{2}}{2}} {\rm d}x.
\end{align}

\paragraph{\ref{Step3} First-passage bridge and identification of the kernel.} By \eqref{eq19}, Remark \ref{Remark5} and Corollary \ref{corollary1}, we have that, for $x > 0$ and for all positive or bounded measurable function $G^{\prime}: \mathbf{C}([0,1], \mathbb{R}) \rightarrow \mathbb{R}$,
\begin{align}
\mathbb{E}[G^{\prime}(\tilde{Z}_{x}^{\rm br})] =  \mathbb{E}[ G^{\prime}(\mathcal{T}^{\rm br} \circ \mathcal{G}^{\rm me}((B(t), t \in [0,1])) ) \mid T_{x} = 1].
\end{align}
\noindent In the event $\{T_{x} = 1\}$, it follows from Lemma \ref{lemma11} and Remark \ref{Remark8} that $\mathcal{T}^{\rm br} \circ \mathcal{G}^{\rm me}((B(t), t \in [0,1])) = (\tilde{X}^{\rm br}_{x}(t), t \in [0,1])$, where $\tilde{X}^{\rm br}_{x}$ is the process defined in \eqref{eq28}. Then, for all positive or bounded measurable function $G^{\prime}: \mathbf{C}([0,1], \mathbb{R}) \rightarrow \mathbb{R}$,
\begin{align} \label{eq20}
\mathbb{E}[G^{\prime}(\tilde{Z}_{x}^{\rm br})] =  \mathbb{E}[ G^{\prime}(\tilde{X}^{\rm br}_{x}) \mid T_{x} = 1].
\end{align}

By \eqref{eq25}, \eqref{eq28} and Remark \ref{Remark9}, observe that
\begin{align} \label{eq30}
\max_{0 \leq t \leq 1} \tilde{X}_{x}^{\rm br}(t) = \frac{x}{2}(\tau_{x})^{-1/2}.
\end{align}
\noindent Recall from Section \ref{sec:lengthofnon-excisedexc} that $\tau_{x}^{\rm e}$ and $\tau_{x}$ are independent. From \eqref{eq17} and \eqref{eq44}, it follows that $\tau_{x}^{\rm e} + \tau_{x} =T_{x}$. Recall also that the probability density function $f_{T_{x}}$ of $T_{x}$ is given by 
\begin{align}
f_{T_{x}}(t) = \frac{x}{\sqrt{2\pi} t^{3/2}} e^{-\frac{x^{2}}{2t}} \mathbf{1}_{\{t >0\}}, \quad \text{for} \quad t \in \mathbb{R};
\end{align}
\noindent see, e.g., \cite[page 107]{Yor1999}. Furthermore, $\tau_{x}^{\rm e}$ is independent of $\tilde{X}_{x}^{\rm br}$, which follows from the Poissonian construction in Section \ref{sec:excise excursion}. Then, by \eqref{eq29}, \eqref{eq20} and \eqref{eq30}, we have that, for all positive or bounded measurable function $G: \mathbf{C}([0,1], \mathbb{R}) \rightarrow \mathbb{R}$,
\begin{align} \label{eq21}
\mathbb{E}[G(\tilde{Z}_{x}^{\rm br}) (\tau_{x}^{\rm br})^{1/2}] & = \mathbb{E}[G(\tilde{X}_{x}^{\rm br}) \tau_{x}^{1/2} \mid \tau_{x}^{\rm e} + \tau_{x}  =1] \nonumber \\
&  = \frac{\sqrt{2\pi}}{x} e^{\frac{x^{2}}{2}} \mathbb{E}[G(\tilde{X}_{x}^{\rm br}) \tau_{x}^{1/2} f_{\tau_{x}^{\rm e}}(1-\tau_{x}) ],
\end{align}
\noindent where $f_{\tau_{x}^{\rm e}}$ is the density of $\tau_{x}^{\rm e}$ defined in \eqref{eq8}. Moreover, by Lemma \ref{lemma4}, \eqref{eq30} and \eqref{eq21},
\begin{align} \label{eq23}
\mathbb{E}[G(\tilde{Z}_{x}^{\rm br}) (\tau_{x}^{\rm br})^{1/2}] = e^{\frac{x^{2}}{2}} \mathbb{E}\left[G(B^{\rm exc})f_{\tau_{x}^{\rm e}}\Big(1-\frac{x^{2}}{4}(\bar{B}^{\rm exc})^{-2} \Big) \right],
\end{align}
\noindent where $(B^{\rm exc}(t), t \in [0,1])$ is a standard Brownian excursion and $\bar{B}^{\rm exc}  = \max_{0 \leq s \leq 1} B^{\rm exc}(s)$. 

On the other hand, by \eqref{eq36} and \eqref{eq8}, we note that
\begin{align}
\phi_x(t)=2 f_{\tau^{\rm e}_x}\!\big(1-(x/2t)^2\big).
\end{align}
\noindent Therefore, substituting \eqref{eq23} into \eqref{eq:A2} yields \eqref{eq37}. This concludes our proof. 
\end{proof}

\begin{appendices}

\section{Proof of Proposition \ref{Pro3}} \label{AProofofPro3}

We first establish two intermediate results prior to proving Proposition \ref{Pro3}.

\begin{lemma} \label{lemma6}
Fix $a > 0$ and let $(w_{n})_{n \geq 1}$ be a sequence of functions in $\mathbf{B}_{a}$ that converges to $w\in \mathbf{B}^{\ast}_a$ with respect to the uniform topology as $n \rightarrow \infty$. Then, $\rho(w_{n}) \rightarrow \rho(w)$, as $n \rightarrow \infty$. Furthermore, $\mathcal{T}^{\rm me}_{a}$ is continuous at every $w\in \mathbf{B}^{\ast}_a$ with respect to the uniform topology. 
\end{lemma}

\begin{proof}[Proof of Lemma \ref{lemma6}]
Since $w\in \mathbf{B}^{\ast}_a$, $w(\rho) > w(s)$ for any $s \in [0,a]$ such that $s \neq \rho = \rho(w)$. In particular, we also have that $\rho \in (0,a)$.  Then, for any $\varepsilon \in (0, \min(\rho, a-\rho))$, let $A_{\varepsilon} \coloneqq [0,\rho - \varepsilon] \cup [\rho+\varepsilon, a]$ and observe that $\max_{s \in A_{\varepsilon}} w_{n}(s) - w_{n}(\rho)  \rightarrow \max_{s \in A_{\varepsilon}} w(s) - w(\rho) < 0$, as $n \rightarrow \infty$. This implies that the maximum of $w_{n}$ on $[0,a]$ is achieved in $[\rho - \varepsilon,  \rho + \varepsilon]$ for large enough $n$, that is, $\rho(w_{n})  \in [\rho - \varepsilon,  \rho + \varepsilon]$ for large enough $n$. The arbitrariness of $\varepsilon$ implies that  $\rho(w_{n}) \rightarrow \rho$, as $n \rightarrow \infty$. This concludes the proof of our first claim.

To prove the second part, for $n \geq 1$, define the bijection $\theta_{n}: [0,a] \rightarrow [0,a]$ by letting $\theta_{n}(0) = 0$, $\theta_{n}(\rho) = \rho(w_{n})$, $\theta_{n}(a) =a$, and then using linear interpolation between these points. By our first claim, $(\theta_{n})_{n \geq 1}$ converges to the identity function uniformly on $[0,a]$, as $n \rightarrow \infty$. Thus, $(\overleftarrow{w}_{n} \circ \theta_{n})_{n \geq 1}$ converges to $\overleftarrow{w}$ with respect to the uniform topology as $n \rightarrow \infty$. This implies our second claim (see, e.g., \cite[Proposition VII.1.17 (b)]{Jacod2003}) and its proof) and concludes our proof.
\end{proof}

\begin{lemma} \label{lemma7}
Fix $a > 0$ and let $(w_{n})_{n \geq 1}$ be a sequence of functions in $\mathbf{M}_a$ that converges to $w\in \mathbf{M}_a^{\ast}$ with respect to the uniform topology as $n \rightarrow \infty$. Then, $\gamma_{w_{n}(a)/2}  \rightarrow \gamma_{w(a)/2}$, as $n \rightarrow \infty$. Furthermore, $\mathcal{T}^{\rm br}_{a}$ is continuous at every $w\in \mathbf{M}_a^{\ast}$ with respect to the uniform topology.
\end{lemma}

\begin{proof}[Proof of Lemma \ref{lemma7}]
Since $w\in \mathbf{M}_a^{\ast}$, we know that $w(0)=0$ and $w(a) >0$ and thus, $\gamma_{w(a)/2}  \in (0,a)$. In particular, $w(a)/2 = w(\gamma_{w(a)/2}) > w(s)$ for any $s \in [0,\gamma_{w(a)/2}]$ such that $s \neq \gamma_{w(a)/2}$. Then, for any $\varepsilon \in (0, \gamma_{w(a)/2})$, let $A_{\varepsilon} \coloneqq [0,\gamma_{w(a)/2} - \varepsilon]$ and observe that $\max_{s \in A_{\varepsilon}} w_{n}(s) - w_{n}(a)/2  \rightarrow \max_{s \in A_{\varepsilon}} w(s) - w(a)/2 < 0$, as $n \rightarrow \infty$. This implies that $w_{n}$ achieves $w_{n}(a)/2$ for the first time in $[\gamma_{w(a)/2}-\varepsilon, a]$ for large enough $n$, that is, $\liminf_{n \rightarrow \infty} \gamma_{w_{n}(a)/2} \geq \gamma_{w(a)/2}$. 

On the other hand, since $w\in \mathbf{M}_a^{\ast}$, $w(a) = \max_{u \in [s, a]}(u)$, for any $s \in [\gamma_{w(a)/2},a]$. Then, for any $\varepsilon \in (0, a-\gamma_{w(a)/2})$, let $B_{\varepsilon} \coloneqq [\gamma_{w(a)/2} + \varepsilon, a]$ and observe that $\max_{s \in B_{\varepsilon}} w_{n}(s) - w_{n}(a)/2  \rightarrow \max_{s \in B_{\varepsilon}} w(s) - w(a)/2 > 0$, as $n \rightarrow \infty$. This implies that $w_{n}$ achieves $w_{n}(a)/2$ for the first time in $[0,\gamma_{w(a)/2}+\varepsilon]$ for large enough $n$, that is, $\limsup_{n \rightarrow \infty} \gamma_{w_{n}(a)/2} \leq \gamma_{w(a)/2}$. By the preceding paragraph, this concludes the proof of our first claim.

The proof of our second claim follows similarly as in the second part of the proof of Lemma \ref{lemma6} with the necessary modifications. Details are left to the reader.
\end{proof}

For a real number $a>0$, let $\| \cdot \|_{a}$ denote the supremum norm on $\mathbf{C}([0,a], \mathbb{R})$, that is, $\| w \|_{a} = \sup_{0\leq t \leq a} |w(t)|$, for $w \in \mathbf{C}([0,a], \mathbb{R})$.

\begin{proof}[Proof of Proposition \ref{Pro3}]
By Lemmas \ref{lemma6} and \ref{lemma7}, it suffices to show that $\mathbf{B}^{\ast}_{a}$ and $\mathbf{M}_a^{\ast}$ are Borel subsets of $\mathbf{C}([0,a], \mathbb{R})$. To this end, let us temporarily assume the following holds:
\begin{enumerate}[label=(\textbf{\roman*})]
\item \label{Set1}  For $u \in [0,a]$, $\mathbf{E}_{u}^{(1)} \coloneqq \{w \in \mathbf{C}([0,a], \mathbb{R}): w(u) = 0\}$ is a closed set of $\mathbf{C}([0,a], \mathbb{R})$.

\item \label{Set2} For $u \in [0,a]$, $\mathbf{E}_{u}^{(2)} \coloneqq \{w \in \mathbf{C}([0,a], \mathbb{R}): \max_{0 \leq s \leq u} w(s) = \max_{u \leq s \leq a} w(s)\}$ is a closed set of $\mathbf{C}([0,a], \mathbb{R})$, 

\item \label{Set3} For $u \in [0,a]$ and $b \in [0,a]$, $\mathbf{E}_{u,b}^{(3)} \coloneqq \{w \in \mathbf{C}([0,a], \mathbb{R}): \max_{0 \leq s \leq b} w(s) > w(u)\}$ is an open set of $\mathbf{C}([0,a], \mathbb{R})$.  

\item \label{Set4} For $u \in [0,a]$ and $b \in [0,a]$, $\mathbf{E}_{u, b}^{(4)} \coloneqq \{w \in \mathbf{C}([0,a], \mathbb{R}): \max_{0 \leq s \leq b} w(s) = w(u)\}$ is a closed set of $\mathbf{C}([0,a], \mathbb{R})$. 
\end{enumerate}

\noindent Then, observe that
\begin{align}
\mathbf{B}_{a}^{\ast}  =   \mathbf{E}_{0}^{(1)} \cap \mathbf{E}_{1}^{(1)}  \cap \bigcap_{q \in [0,a] \cap \mathbb{Q}} (\mathbf{E}_{q}^{(2)} \cap \mathbf{E}_{q,a}^{(3)})^{c}
\end{align}
\noindent and 
\begin{align}
\mathbf{M}_a^{\ast} =  \mathbf{E}_{0}^{(1)} \cap \bigcap_{q \in [0,a) \cap \mathbb{Q}}  (\mathbf{E}_{a,q}^{(3)} \cup \mathbf{E}_{a,q}^{(4)})^{c}.
\end{align}
\noindent Therefore, $\mathbf{B}^{\ast}_{a}$ and $\mathbf{M}_a^{\ast}$ are Borel subsets of $\mathbf{C}([0,a], \mathbb{R})$.

Let us now prove \ref{Set1} through \ref{Set4}. We start with the proof of \ref{Set1}. Let $w \in \mathbf{C}([0,a], \mathbb{R})$ be an accumulation point of $\mathbf{E}_{u}^{(1)}$. Then, for every $\varepsilon > 0$, there exists $w_{\varepsilon} \in \mathbf{E}_{u}^{(1)}$ such that $\| w - w_{\varepsilon} \|_{a} < \varepsilon$. In particular, $|w(u) - w_{\varepsilon}(u)| = |w(u)| < \varepsilon$, that is, $w(u)= 0$. Thus, $w \in \mathbf{E}_{u}^{(1)}$, which implies \ref{Set1}. 

Next, we prove \ref{Set2}. Let $w \in \mathbf{C}([0,a], \mathbb{R})$ be an accumulation point of $\mathbf{E}_{u}^{(2)}$. Then, for every $\varepsilon > 0$, there exists $w_{\varepsilon} \in \mathbf{E}_{u}^{(2)}$ such that $\| w - w_{\varepsilon} \|_{a} < \varepsilon/2$. In particular, $|\max_{0 \leq s \leq u}w(s) - \max_{0 \leq s \leq u}w_{\varepsilon}(s)| < \varepsilon/2$ and $|\max_{u \leq s \leq a}w(s) - \max_{u \leq s \leq a}w_{\varepsilon}(s)| < \varepsilon/2$. Then, for all $\varepsilon >0$, $|\max_{0 \leq s \leq u}w(s) - \max_{u \leq s \leq a} w(s)| < \varepsilon$, that is, $w \in \mathbf{E}_{u}^{(2)}$. This implies \ref{Set2}. 

We continue with the proof of \ref{Set3}. For $w \in \mathbf{E}_{u,b}^{(3)}$, set $\varepsilon = \max_{0 \leq s \leq b} w(s) - w(u) >0$. Then, for $w_{\varepsilon} \in \mathbf{C}([0,a], \mathbb{R})$ such that $\| w - w_{\varepsilon} \|_{a} < \varepsilon/3$, we have that
\begin{align*}
\left|\max_{0 \leq s \leq b} w(s) - w(u) - \max_{0 \leq s \leq b} w_{\varepsilon} (s) + w_{\varepsilon}(u) \right| < \frac{2\varepsilon}{3}.
\end{align*}
\noindent In particular,
\begin{align*}
\max_{0 \leq s \leq b} w_{\varepsilon} (s) - w_{\varepsilon}(u)  > -\frac{2\varepsilon}{3} + \max_{0 \leq s \leq b} w(s) - w(u) >0.
\end{align*} 
\noindent Therefore, $w_{\varepsilon} \in \mathbf{E}_{u, b}^{(3)}$, which implies \ref{Set3}. 

Finally, we prove \ref{Set4}. Let $w \in \mathbf{C}([0,a], \mathbb{R})$ be an accumulation point of $\mathbf{E}_{u,b}^{(4)}$. Then, for every $\varepsilon > 0$, there exists $w_{\varepsilon} \in \mathbf{E}_{u,b}^{(4)}$ such that $\| w - w_{\varepsilon} \|_{a} < \varepsilon$. In particular, $| \max_{0 \leq s \leq b} w(s) - \max_{0 \leq s \leq b} w_{\varepsilon}(s) | < \varepsilon/2$ and similarly, $|w(u) - w_{\varepsilon}(u)| < \varepsilon/2$.  Then, for all $\varepsilon >0$, $|\max_{0 \leq s \leq b}w(s) - w(u)| < \varepsilon$, that is, $w \in \mathbf{E}_{u,b}^{(4)}$. This implies \ref{Set4}. 
\end{proof}

\section{Proof of Proposition \ref{Pro2}} \label{AProofofPro2}

Before proving Proposition \ref{Pro2}, we provide some preliminary remarks and intermediate results that will be used in its proof.

\begin{remark} \label{Remark4}
Observe that $\mathbf{B}^{{\rm reg}}_1 \subset \mathbf{B}^{\ast}_1$. In particular, if  $w\in \mathbf{B}^{{\rm reg}}_1$ and $\rho$ is the unique location of its maximum on the interval $[0,1]$, then it follows from \ref{RegPro1} in Definition \ref{DefinitionRegular} that $\vec{\rho}_{t} = \rho$ for $t \in (\rho, 1] \cap \mathbb{Q}$, and $\protect\cev{\rho}_{t} = \rho$ for $t \in [0, \rho) \cap \mathbb{Q}$. 
\end{remark}

\begin{remark} \label{remark3}
Observe that, for $w\in \mathbf{B}^{{\rm reg}}_1$ and $t \in [0,1]$,
\begin{align}
r_{t}^{\rm br}(w) = \max_{g_{t}^{\rm br}(w) \leq s \leq d_{t}^{\rm br}(w)} m^{\rm br}_{s}(w) \quad \Longleftrightarrow \quad \min_{g_{t}^{\rm br}(w) \leq s \leq d_{t}^{\rm br}(w)} w(s) = 0. 
\end{align}
\noindent In particular, by \ref{RegPro2} in Definition \ref{DefinitionRegular}, $0$ is not a strict local minimum for $w$. Thus, for every $t \in [0,1]$, either
\begin{align}
r_{t}^{\rm br}(w) < \max_{g_{t}^{\rm br}(w) \leq s \leq d_{t}^{\rm br}(w)} m^{\rm br}_{s}(w) \quad \text{(kept with margin)}, 
\end{align}
\noindent or 
\begin{align}
r_{t}^{\rm br}(w) > \max_{g_{t}^{\rm br}(w) \leq s \leq d_{t}^{\rm br}(w)} m^{\rm br}_{s}(w) \quad\text{(excised with margin)}.
\end{align}
\noindent This means that no excursion of a regular function grazes the threshold $0$; that is, each excursion is either strictly retained or strictly excised. In particular, for $w \in \mathbf{B}^{{\rm reg}}_{1}$, $\mathcal{G}^{\rm br}(w)$ is strictly positive on $(0,1)$ and is zero only at the endpoints of $[0,1]$. 
\end{remark}

\begin{lemma} \label{lemma10} 
$\mathcal{G}^{\rm br}: \mathbf{B}_1 \to  \mathbf{C}([0,1], \mathbb{R})$ is continuous at every $w \in \mathbf{B}^{\rm reg}_1$ with respect to the uniform topology. 
\end{lemma}

Before proving Lemma \ref{lemma10}, we prove the following two lemmas. In particular, we show the continuity of the left and right endpoints of the excursion of $m^{\rm br}(w)- w$ that straddle time $t \in [0, 1]$. 

\begin{lemma}
For $w\in \mathbf{B}_1$, the map $t \in [0,1] \mapsto g_{t}^{\rm br}(w)$ is right-continuous. 
\end{lemma}

\begin{proof}
Observe that $t \in [0,1] \mapsto g_{t}^{\rm br}(w)$ is non-decreasing. Fix $t \in [0,1]$ and let $(t_{n})_{n \geq 1}$ be an arbitrary sequence of real numbers in $[0,1]$, such that $t_{n} \downarrow t$, as $n \rightarrow \infty$. If there exists $n \geq 1$ such that $g_{t_{n}}^{\rm br}(w) < t$, then $g_{t_{m}}^{\rm br}(w) = g_{t}^{\rm br}(w)$ for $m \geq n$. Otherwise, if $g_{t_{n}}^{\rm br}(w) \geq t$ for every $n \geq 1$, then $t \leq g_{t_{n}}^{\rm br}(w) \leq t_{n}$ for every $n \geq 1$, hence $t$ is itself satisfies $w(t) = m_{t}^{\rm br}(w)$ and $g_{t}^{\rm br}(w) = t = \lim_{n \rightarrow \infty} g_{t_{n}}^{\rm br}(w)$. 
\end{proof}

For $w\in \mathbf{B}_1$, let $J(g^{\rm br}(w))$ be the set of discontinuities of the map $t \in [0,1] \mapsto g_{t}^{\rm br}(w)$; recall that this set is countable.

\begin{lemma} \label{lemma5}
Let $(w_{n})_{n \geq 1}$ be a sequence of functions in $\mathbf{B}_1$ that converges to $w\in \mathbf{B}_{1}$ with respect to the uniform topology as $n \rightarrow \infty$. Suppose that $w$ satisfies property \ref{RegPro1} in Definition \ref{DefinitionRegular}. Then, for every $t \in [0,1]$, $d_{t}^{\rm br}(w_{n}) \rightarrow d_{t}^{\rm br}(w)$, as $n \rightarrow \infty$. Furthermore, for every $t \in [0,1] \setminus ((0,1) \cap J(g^{\rm br}(w)))$, $g_{t}^{\rm br}(w_{n}) \rightarrow g_{t}^{\rm br}(w)$, as $n \rightarrow \infty$. 
\end{lemma}

\begin{proof}
To simplify the notation, for $n \geq 1$ and $t \in [0,1]$, we set $g_{t,n}^{\rm br} = g_{t}^{\rm br}(w_{n})$ and $d_{t,n}^{\rm br} = d_{t}^{\rm br}(w_{n})$. We also write $g_{t}^{\rm br} = g_{t}^{\rm br}(w)$ and $d_{t}^{\rm br} = d_{t}^{\rm br}(w)$. Since $w$ satisfies property \ref{RegPro1} of Definition \ref{DefinitionRegular}, it follows that $w \in \mathbf{B}_1^{\ast}$. Then, let $\rho = \rho(w)$ be the unique location of its maximum on the interval $[0,1]$. 

{\bf Converge of the right endpoints}.  If $t = 1$, then $d_{1,n}^{\rm br}=1 \rightarrow d_{1}^{\rm br}=1$, as $n \rightarrow \infty$. Suppose that $t \in [0,  \rho]$. Indeed, since $t \in [0,  \rho]$, we necessarily must have $d_{t}^{\rm br} \in [t, \rho]$ (due to the uniqueness of the location of the maximum of $w$). The case $t \in [\rho, 1]$ follows similarly, and we omit the details.

Let us consider the case $t = \rho$. Note that we must necessarily have $\rho <1$. Note that $d_{t}^{\rm br} = \rho$ and $m^{\rm br}_{\rho}(w) = w(d_{t}^{\rm br}) > w(s)$ for any $s \in [\rho,1]$ such that $s \neq d_{t}^{\rm br}$. Then, for any $\varepsilon \in (0, 1-\rho)$. By our convergence assumption, $\max_{s \in [\rho + \varepsilon, 1]} w_{n}(s) - w_{n}(\rho)  \rightarrow \max_{s \in [\rho + \varepsilon, 1]} w(s) - w(\rho) < 0$, as $n \rightarrow \infty$. Therefore, the maximum of $w_{n}$ on $[\rho,1]$ is achieved in $[\rho ,\rho + \varepsilon]$ for large enough $n$, that is, $d_{t,n}^{\rm br} \in [\rho ,\rho + \varepsilon]$ for large enough $n$. The arbitrariness of $\varepsilon$ implies that $d_{t,n}^{\rm br} \rightarrow \rho = d_{t}^{\rm br}$, as $n \rightarrow \infty$.

We now suppose that $t \in [0,  \rho)$ and thus, by the uniqueness of $\rho$, we have that $d_{t}^{\rm br} \in [t, \rho)$.  First, we prove that $\liminf_{n \rightarrow \infty} d_{t,n}^{\rm br} \geq d_{t}^{\rm br}$. In the case when $t = d_{t}^{\rm br}$, this is immediate since $d_{t,n}^{\rm br} \geq t = d_{t}^{\rm br}$. We consider the case $d_{t}^{\rm br} \in (t, \rho)$. Note that for any $\varepsilon \in (0, d_{t}^{\rm br}-t)$, we have that, by the definition of $d_{t}^{\rm br}$ (see \eqref{eq42}), $m_{t}^{\rm br}(w_{n}) - \max_{s \in [t, d_{t}^{\rm br} -\varepsilon]} w_{n}(s)   \rightarrow m_{t}^{\rm br}(w) - \max_{s \in [t, d_{t}^{\rm br} -\varepsilon]} w(s) > 0$, as $n \rightarrow \infty$. We deduce that $d_{t,n}^{\rm br} \geq d_{t}^{\rm br} - \varepsilon$, for $n$ large enough, that is, $\liminf_{n \rightarrow \infty} d_{t,n}^{\rm br} \geq d_{t}^{\rm br}$. 

We now prove that $\limsup_{n \rightarrow \infty} d_{t,n}^{\rm br} \leq d_{t}^{\rm br}$. Recall that  $d_{t}^{\rm br} \in [t, \rho)$. If $t < d_{t}^{\rm br}$, then  \ref{RegPro1} in Definition \ref{DefinitionRegular} implies that, for every $\varepsilon \in (0, \rho- d_{t}^{\rm br}]$, $m_{d_{t}^{\rm br} + \varepsilon}^{\rm br}(w) >m_{d_{t}^{\rm br}}^{\rm br}(w) = m_{t}^{\rm br}(w)$. On the other hand, if there exists $\varepsilon \in (0, \rho- d_{t}^{\rm br})$ such that  $d_{t,n}^{\rm br} >  d_{t}^{\rm br} + \varepsilon$ for large $n$ then $m_{t}^{\rm br}(w_{n}) = m_{d_{t,n}^{\rm br}}^{\rm br}(w_{n})$, and so we get a contradiction
\begin{align}
m_{t}^{\rm br}(w) = \lim_{n \rightarrow \infty} m_{t}^{\rm br}(w_{n}) = \lim_{n \rightarrow \infty}  m_{d_{t,n}^{\rm br}}^{\rm br}(w_{n}) \geq \lim_{n \rightarrow \infty} m_{d_{t}^{\rm br} + \varepsilon}^{\rm br}(w_{n}) = m_{d_{t}^{\rm br} + \varepsilon}^{\rm br}(w) > m_{t}^{\rm br}(w).
\end{align}
\noindent We conclude that $\limsup_{n \rightarrow \infty} d_{t,n}^{\rm br} \leq d_{t}^{\rm br}$. 

If $t = d_{t}^{\rm br}$, then $d_{t}^{\rm br} = g_{t}^{\rm br}$ and $m_{d_{t}^{\rm br}}^{\rm br}(w) = m_{t}^{\rm br}(w)$. Furthermore, two cases may arise:
\begin{enumerate}[label=(\textbf{\alph*})]
  \item \label{CaseIRight} There exists $\delta \in (0, \rho- t]$ such that $w(s) =  m_{t}^{\rm br}(w)$, for all $s \in [t, t+\delta]$;
  \item \label{CaseIIRight} There exists $\delta \in (0, \rho- t]$ such that $w(s) >  m_{t}^{\rm br}(w)$, for all $s \in (t, t+\delta]$.
\end{enumerate}
\noindent Case \ref{CaseIRight} is ruled out by \ref{RegPro1} of Definition \ref{DefinitionRegular}. In case \ref{CaseIIRight}, $\limsup_{n \rightarrow \infty} d_{t,n}^{\rm br} \leq d_{t}^{\rm br}$ follows by the same argument used for $t < d_{t}^{\rm br}$. This finished our proof of the converge of the right endpoints. 

{\bf Converge of the left endpoints}. If $t = 0$, then $g_{0,n}^{\rm br}=0 \rightarrow g_{0}^{\rm br}=0$, as $n \rightarrow \infty$. Suppose that $t \in (0, \rho] \cap \mathbb{Q}$. Then, by \ref{RegPro1} in Definition \ref{DefinitionRegular}, $g_{t}^{\rm br} = \vec{\rho}_{t} \in [0, t]$ and $m^{\rm br}_{t}(w) = w(g_{t}^{\rm br}) > w(s)$ for any $s \in [0,t]$ such that $s \neq g_{t}^{\rm br}$. Suppose that $g_{t}^{\rm br} < t$. Then, for any $\varepsilon \in (0, \min(g_{t}^{\rm br}, t- g_{t}^{\rm br}))$, let $A_{t, \varepsilon} \coloneqq [0,g_{t}^{\rm br} - \varepsilon] \cup [g_{t}^{\rm br}+\varepsilon, t]$. By our convergence assumption, $\max_{s \in A_{t, \varepsilon}} w_{n}(s) - w_{n}(g_{t}^{\rm br})  \rightarrow \max_{s \in A_{t, \varepsilon}} w(s) - w(g_{t}^{\rm br}) < 0$, as $n \rightarrow \infty$. Therefore, the maximum of $w_{n}$ on $[0,t]$ is achieved in $[g_{t}^{\rm br} - \varepsilon,  g_{t}^{\rm br} + \varepsilon]$ for large enough $n$, that is, $g_{t,n}^{\rm br} \in [g_{t}^{\rm br} - \varepsilon,  g_{t}^{\rm br} + \varepsilon]$ for large enough $n$. Then, $g_{t,n}^{\rm br} \rightarrow g_{t}^{\rm br}$, as $n \rightarrow \infty$. If $t = g_{t}^{\rm br}$, the proof follows similarly by considering $\varepsilon \in (0, g_{t}^{\rm br})$ and letting $A_{t, \varepsilon} \coloneqq [0,g_{t}^{\rm br} - \varepsilon]$. Finally, the case $t \in [\rho, 1] \cap \mathbb{Q}$ follows along the same lines, and the details are left to the reader.

We now prove the convergence of the left endpoints for all $t \in (0,\rho]$. Again, the case $t \in [\rho, 1)$ follows similarly, and we omit the details. Suppose that $g_{t}^{\rm br}\leq t< d_{t}^{\rm br}$. Then, for any $s \in (t, d_{t}^{\rm br}) \cap \mathbb{Q}$, we have that $g_{s,n}^{\rm br} \rightarrow g_{s}^{\rm br}=g_{t}^{\rm br}$ and $d_{s,n}^{\rm br} \rightarrow d_{s}^{\rm br}=d_{t}^{\rm br}$, as $n \rightarrow \infty$. Since $d_{t,n}^{\rm br}\rightarrow d_{t}^{\rm br}$, as $n \rightarrow \infty$, and $g_{t}^{\rm br}<s$, we have that $g_{t,n}^{\rm br}  \leq g_{s,n}^{\rm br} \leq s < d_{t,n}^{\rm br}$, for large enough $n$. Thus, $g_{t,n}^{\rm br} = g_{s,n}^{\rm br} \rightarrow g_{t}^{\rm br}$, as $n \rightarrow \infty$.

Finally, if $g_{t}^{\rm br}=d_{t}^{\rm br}$, then $t = g_{t}^{\rm br}$ and $m_{t}^{\rm br}(w) = w(t)$. Furthermore, two cases may arise:
\begin{enumerate}[label=(\textbf{\alph*})]
\item \label{CaseIILeft} There exists $\delta \in (0, t]$ such that $w(s) = m_{t}^{\rm br}(w)$, for all $s \in [t-\delta, t]$;
    \item \label{CaseIIILeft} There exists $\delta \in (0, t]$ such that $w(s) <  m_{t}^{\rm br}(w)$, for all $s \in [t-\delta, t)$.
\end{enumerate}
\noindent Case \ref{CaseIILeft} is ruled out by \ref{RegPro1} of Definition \ref{DefinitionRegular}. Thus, we only need to consider case \ref{CaseIIILeft}. If $t$ is the unique point in $[0, t]$ such that $m_{t}^{\rm br}(w) = w(t)$ (this also includes the case $t = \rho$), then our claim follows by the same argument use for rational $t$. Suppose there exists $t^{\prime} \in [0, t)$ such that $w(t) = m_{t}^{\rm br}(w) = w(t^{\prime})$. Indeed, by \ref{RegPro1} of Definition \ref{DefinitionRegular}, $t^{\prime}$ must be unique. However, the above implies that $t \in J(g^{\rm br}(w))$, which is ruled out by our assumption. 

This finished our proof of the converge of the left endpoints. 
\end{proof}

\begin{proof}[Proof of Lemma \ref{lemma10}]
Let $(w_{n})_{n \geq 1}$ be a sequence of functions in $\mathbf{B}_1$ that converges to $w\in \mathbf{B}^{{\rm reg}}_1$ with respect to the uniform topology as $n \rightarrow \infty$. Suppose momentarily that we have proven:
\begin{enumerate}[label=(\textbf{\roman*})]
\item \label{Conv1}  $(m^{\rm br}(w_n))_{n \geq 1}$ converges to $m^{\rm br}(w)$ uniformly on $[0,1]$, as $n \rightarrow \infty$. In particular, $(m^{\rm br}(w_n) - w_{n})_{n \geq 1}$ converges to $m^{\rm br}(w)-w$ uniformly on $[0,1]$, as $n \rightarrow \infty$;

\item \label{Conv2} For every $t \in [0,1] \setminus ((0,1) \cap J(g^{\rm br}(w)))$, $\mathbf{1}_{\{r^{\rm br}_{t}(w_{n})<m^{\rm br}_{t}(w_{n})\}} \rightarrow \mathbf{1}_{\{r^{\rm br}_{t}(w)<m^{\rm br}_{t}(w)\}}$, as $n \rightarrow \infty$;

\item \label{Conv3} 
$(u^{\rm br}(w_{n}))_{n \geq 1}$ converges to $u^{\rm  br}(w)$ uniformly on $[0,1]$, as $n \rightarrow \infty$;

\item \label{Conv4} For $n \geq 1$ and $t \in [0,1]$, define $u_{1,n}^{\rm br} \coloneqq u_{1}^{\rm br}(w_{n})$ and $u_{1}^{\rm br} \coloneqq u_{1}^{\rm br}(w)$. To simplify the notation, we also write $h_{t}^{(n)} \coloneqq \alpha^{\rm br}_{u_{1,n}^{\rm br} t}(w_{n})$ and $h_{t} \coloneqq \alpha^{\rm br}_{u_{1}^{\rm br} t}(w)$. Then, $(h^{(n)})_{n \geq 1}$ converges to $h$ in $\mathbf{D}([0,1], \mathbb{R})$, where $\mathbf{D}([0,1], \mathbb{R})$ is the space of real-valued c\`adl\`ag functions on $[0,1]$ equipped with the Skorohod ${\rm J}_{1}$ topology; see e.g., \cite[Chapter 3]{Bi1999} or \cite[Chapter VI]{Jacod2003}. 
\end{enumerate} 

\noindent We show how \ref{Conv3} and \ref{Conv4} imply our claim. First, recall that, $u^{\rm br}_{1} >0$. We may therefore choose $N \in \mathbb{N}$ sufficiently large such that $u^{\rm br}_{1,n}>0$ for all $n \geq N$. Then, for $n \geq N$,
\begin{align}
\mathcal{G}^{\rm br}(w_{n}) = ( (u^{\rm br}_{1,n})^{-1/2} w_{n}(h_{t}^{(n)}), t \in [0,1]) \quad \text{and} \quad
\mathcal{G}^{\rm br}(w) = ( (u^{\rm br}_{1})^{-1/2} w(h_{t}), t \in [0,1]).
\end{align}
\noindent Therefore our claim follows from \ref{Conv3}, \ref{Conv4} and \cite[Theorem 3.1]{Whitt1980} (see also the proof of the Lemma in \cite[pp.\ 151-152]{Bi1999}). 

We now proceed to prove \ref{Conv1} through \ref{Conv4}. Clearly, \ref{Conv1} follows from the uniform convergence of $(w_{n})_{n \geq 1}$ towards $w$. Recall that we have assumed that $w$ satisfies \ref{RegPro1} in Definition \ref{DefinitionRegular}.  Then, \ref{Conv1}, Lemma \ref{lemma5}, and the continuity of $w$ imply that for every $t \in [0,1]$, $r^{\rm br}_{t}(w_{n}) \rightarrow r^{\rm br}_{t}(w)$, as $n \rightarrow \infty$. (Recall that $J(g^{\rm br}(w))$ has Lebesgue measure $0$.) This, together with the assumption $w \in \mathbf{B}^{{\rm reg}}_1$ (specifically \ref{RegPro2} in Definition \ref{DefinitionRegular}; see also Remark \ref{remark3}), implies \ref{Conv3}.

Finally, we prove \ref{Conv4}. For $f \in \mathbf{D}([0,1], \mathbb{R})$, we set $\Delta f_{t} \coloneqq f_{t}-f_{t-}$, for $t \in [0,1]$, with the convention $f_{0} = f_{0-}$. Let $J(h) = \{ t \in (0,1]: |\Delta h_{t}| >0 \}$ be the set of all jump times of $h$. Recall that $J(h)$ is at most countable and that $[0,1] \setminus J(h)$ is dense in $[0,1]$. Then, \ref{Conv4} follows from \cite[Theorem VI.2.15]{Jacod2003}, provided that we prove that for all $t \in [0,1] \setminus J(h)$,
\begin{align} \label{eq39}
h^{(n)}_{t} \rightarrow h_{t}, \quad \text{as} \, \, n \rightarrow \infty,
\end{align}
\noindent and
\begin{align} \label{eq40}
\sum_{0 < s \leq t} (\Delta h^{(n)}_{t})^{2} \rightarrow \sum_{0 < s \leq t} (\Delta h_{t})^{2}, \quad \text{as} \, \, n \rightarrow \infty.
\end{align}

\noindent Let us start with the proof of \eqref{eq39}. Fix $t \in [0,1] \setminus J(h)$. Observe from \eqref{eq38} that $u^{\rm br}(w)$ is a non-decreasing, continuous function. Note also from \eqref{eq41} that $h$ is strictly increasing. 

If $h_{t} = 1$, then $\limsup_{t \rightarrow \infty} h_{t}^{(n)} \leq h_{t}$. Next, suppose that $h_{t} \in [0,1)$. Then, for any $\varepsilon \in (0, 1-h_{t})$, there exists $s \in [h_{t}, h_{t} + \varepsilon]$ such that $u^{\rm br}_{s}(w) > u^{\rm br}_{1}(w) t$. But then, \ref{Conv3} implies that $u^{\rm br}_{s}(w_{n}) > u^{\rm br}_{1}(w_{n}) t$ for large enough $n$. Thus, we deduce that  $\limsup_{n \rightarrow \infty} h_{t}^{(n)} \leq h_{t}$. 

If $h_{t} = 0$, then $\liminf_{n \rightarrow \infty} h_{t}^{(n)} \geq h_{t}$, and \eqref{eq39} follows by the preceding paragraph. Otherwise, suppose that $h_{t} \in (0,1]$. Then, since $t \in [0,1] \setminus J(h)$, we have that for any $\varepsilon \in (0, h_{t})$ and for any $s\in [0, h_{t} - \varepsilon]$, $u^{\rm br}_{s}(w) < u^{\rm br}_{1}(w) t$. Hence, there exists $\eta > 0$ such that $|u^{\rm br}_{1}(w) t - u^{\rm br}_{s}(w)| > \eta$ for all $s\in [0, h_{t} - \delta]$. (Recall that $u^{\rm br}(w)$ is non-decreasing and continuous.) But then, it follows from \ref{Conv3} that there exists $N \in \mathbb{N}$ such that  $|u^{\rm br}_{1}(w_{n}) t - u^{\rm br}_{s}(w_{n})| > \eta$ for all $s\in [0, h_{t} - \varepsilon]$ and $n \geq N$. (Note also that $u^{\rm br}(w_{n})$ is non-decreasing and continuous.) This implies that $h_{t}^{(n)} \geq h_{t} - \varepsilon$ for $n \geq N$ and we deduce that $\liminf_{n \rightarrow \infty} h_{t}^{(n)} \geq h_{t}$. Therefore, \eqref{eq39} follows by the preceding paragraph.

Next, we prove \eqref{eq40}. Fix $t \in [0,1] \setminus J(h)$. Observe that,
\begin{align}
\sum_{0 < s \leq t} (\Delta h^{(n)}_{t})^{2} = 2 \int_{0}^{h^{(n)}_{t}} s \mathbf{1}_{\{r^{\rm br}_{s}(w_{n}) \geq m^{\rm br}_{s}(w_{n})\}} {\rm d} s,
\end{align}
\noindent and 
\begin{align}
\sum_{0 < s \leq t} (\Delta h_{t})^{2} = 2 \int_{0}^{h_{t}} s \mathbf{1}_{\{r^{\rm br}_{s}(w) \geq m^{\rm br}_{s}(w)\}} {\rm d} s.
\end{align}
\noindent Then, \eqref{eq40} follows from \ref{Conv2} and \eqref{eq39}. 
\end{proof}

\begin{proof}[Proof of Proposition \ref{Pro2}]
By Lemma \ref{lemma10}, it suffices to show that $\mathbf{B}^{{{\rm reg}}}_1$ is a Borel subset of $\mathbf{C}([0,1], \mathbb{R})$.

For $a \in [0,1]$, let $\mathbf{B}_{[0,a]}^{\ast}$ denote the subset of functions of $\mathbf{B}_{1}$ whose restriction to $[0,a]$ has a unique absolute maximum, and let $\mathbf{B}_{[a,1]}^{\ast}$ denote the subset of functions of $\mathbf{B}_{1}$ whose restriction to $[a,1]$ has a unique absolute maximum.  A similar argument to the one used in the proof of Proposition \ref{Pro3} shows that $\mathbf{B}_{[0,a]}^{\ast}$ and $\mathbf{B}_{[a,1]}^{\ast}$ are Borel subsets of $\mathbf{C}([0,1], \mathbb{R})$. Thus, $\bigcap_{q \in [0,1] \cap \mathbb{Q}} (\mathbf{B}_{[0,q]}^{\ast} \cap \mathbf{B}_{[q,1]}^{\ast})$, 
which is the set of functions in $\mathbf{B}_{1}$ satisfying Definition \ref{DefinitionRegular} \ref{RegPro1}, is also a Borel subset of $\mathbf{C}([0,1], \mathbb{R})$. 

For $q_{1}, q_{2} \in [0,1]$ such that $q_{1} \leq q_{2}$, let $\mathbf{B}_{q_{1}, q_{2}} \coloneqq \{ w \in \mathbf{B}_{1}: \min_{q_{1} \leq s \leq q_{2}} w(s) \neq 0\}$. Define also $I \coloneqq \{ (a,b): a, b \in [0,1] \cap \mathbb{Q} \, \, \text{such that} \, \, a \leq b \, \, \text{and either} \, \, a \in (0,1) \, \,  \text{or} \, \, b \in (0,1)\}$. It is not difficult to prove that, for $q_{1}, q_{2} \in I$, $\{ w \in \mathbf{B}_{1}: \min_{q_{1} \leq s \leq q_{2}} w(s) = 0\}$, is a closed set (for example, by using a similar argument as in the proof of Proposition \ref{Pro3}). Thus, $\bigcap_{(q_{1}, q_{2})  \in I} \mathbf{B}_{q_{1}, q_{2}}$, 
which is the set of functions in $\mathbf{B}_{1}$ satisfying Definition \ref{DefinitionRegular} \ref{RegPro2}, is also a Borel subset of $\mathbf{C}([0,1], \mathbb{R})$.

Therefore, from the preceding two paragraphs, it follows that $\mathbf{B}^{ \rm reg}_1$ is a Borel subset of $\mathbf{C}([0,1], \mathbb{R})$.
\end{proof}

\section{Proof of Proposition \ref{Pro4}} \label{AProofofPro4}

\begin{proof}[Proof of Proposition \ref{Pro4}]
The first claim follows from Propositions \ref{Pro3}, Lemma \ref{lemma9}, Proposition \ref{Pro2}. To establish the second claim, it suffices to show that $\mathbf{M}^{{{\rm reg}}}_1$ is a Borel subset of $\mathbf{C}([0,1], \mathbb{R})$. Observe that, by Remark \ref{Remark9}, $\mathbf{M}^{{\rm reg}}_1 \subset \mathbf{M}^{\ast}_1$, and recall $\mathbf{B}^{{\rm reg}}_1$ is a Borel subset of $\mathbf{C}([0,1], \mathbb{R})$ (see the proof of Proposition \ref{Pro2}). Then, by Definition \ref{DefinitionRegularII} and Proposition~\ref{Pro3}, $\mathbf M^{\mathrm{reg}}_1$ is a Borel subset of $\mathbf C([0,1],\mathbb R)$. 
\end{proof}
\end{appendices}

\paragraph{Acknowledgements.} 
The question studied in this paper was suggested to the second author by Jim Pitman
in a discussion. He emphasized his interest in the interval partition
generated by the excursion landscape, motivated by a visual analogy with
rice paddies. We thank him for this stimulating perspective.


\providecommand{\bysame}{\leavevmode\hbox to3em{\hrulefill}\thinspace}
\providecommand{\MR}{\relax\ifhmode\unskip\space\fi MR }
\providecommand{\MRhref}[2]{%
  \href{http://www.ams.org/mathscinet-getitem?mr=#1}{#2}
}
\providecommand{\href}[2]{#2}

\end{document}